\documentclass[12pt]{amsart}

\pdfoutput=1

\usepackage{amsthm}
\usepackage{amsmath}
\usepackage{amstext}
\usepackage{amssymb}
\usepackage{bbm}
\usepackage{caption}
\usepackage{verbatim} 
\usepackage{setspace}
\usepackage[dvips,letterpaper,margin= 1.3 in]{geometry}

\usepackage{graphicx}
\usepackage{hyperref}
\usepackage{pdflscape}
\usepackage{subcaption}

\usepackage{float}
\usepackage[utf8]{inputenc}
\usepackage[final]{pdfpages}

\newtheorem{theorem}{Theorem}[section]

\newtheorem{lemma}[theorem]{Lemma}
\newtheorem{remark}[theorem]{Remark}
\newtheorem{proposition}[theorem]{Proposition}
\newtheorem{corollary}[theorem]{Corollary}
\newtheorem{conjecture}[theorem]{Conjecture}
\newtheorem{definition}[theorem]{Definition}

\newtheorem{question}[theorem]{Question}
\newtheorem{algorithm}[theorem]{Algorithm}
\newtheorem{observation}[theorem]{Observation}

\newcommand{\dotr}{\mbox{$\boldsymbol{\cdot}$}}

\title{Conditions for Discrete Equidecomposability of Polygons}

\begin{document}

\author{Paxton Turner}
\address{Department of Mathematics, Louisiana State University, Baton Rouge, Louisiana 70803}
\email{pturne7@tigers.lsu.edu}

\author{Yuhuai (Tony) Wu}
\address{Department of Mathematics, University of New Brunswick, Fredericton, New Brunswick E3B4N9}
\email{wu.yuhuai0206@unb.ca}

\thanks{This research was supported by an NSF grant to Brown's Summer@ICERM program.}

\begin{abstract}

Two rational polygons $P$ and $Q$ are said to be \emph{discretely equidecomposable} if there exists a piecewise affine-unimodular bijection (equivalently, a piecewise affine-linear bijection that preserves the integer lattice $\mathbb{Z} \times \mathbb{Z}$) from $P$ to $Q$. In \cite{turnerwu}, we developed an invariant for rational finite discrete equidecomposability known as \textit{weight}. Here we extend this program with a necessary and sufficient condition for rational finite discrete equidecomposability. We close with an algorithm for detecting and constructing equidecomposability relations between rational polygons $P$ and $Q$.  

\end{abstract}

\maketitle

\tableofcontents

\begin{section}{Introduction}
\label{sec:introduction}

Given polygons $P$ and $Q$, a discrete equidecomposability relation $\mathcal{F}:P \to Q$ is a certain type of piecewise-linear $\mathbb{Z} \times \mathbb{Z}$-preserving bijection \cite{haase, kantor1, greenberg, stanley}. Discrete equidecomposability was initially studied due to questions from Ehrhart theory, the enumeration of integer lattice points in dilations of polytopes \cite{ehrhart}. The map $\mathcal{F}$ preserves the Ehrhart quasi-polynomial, which partially explains why discrete equidecomposability relations are natural to consider in the context of Ehrhart theory. Furthermore, discrete equidecomposability has provided a method for shedding light on the mysterious phenomenon of \textit{period collapse}, which occurs when the Ehrhart quasi-polynomial of a rational polygon has a surprisingly simple form \cite{woods, mcallister, deloera, deloera1, derksen, kirillov}.  

We continue the study of the special case of \textit{rational finite discrete equidecomposability} of polygons begun in \cite{turnerwu}. It is our hope that this paper will provide a framework for further exploration of certain computational questions related to discrete equidecomposability.  

\begin{subsection}{Results}

The paper \cite{haase} presented several interesting questions that motivated this research. We present responses to two of these in the dimension $2$ case of rational discrete equidecomposability.

\begin{question}
\label{que:dehn}
Can we develop an invariant for discrete equidecomposability that detects precisely when two rational polytopes $P$ and $Q$ are discretely equidecomposable?\footnote{ This question appears with slightly different terminology as Question 4.5 in \cite{haase}. The authors refer to this invariant as a \emph{discrete Dehn invariant}, alluding to the classical Dehn invariant developed for the question of \emph{continuous equidecomposability}, that is, Hilbert's Third Problem (see Chapter 7 of \cite{ziegler}).}

\end{question}

In the case of finite rational equidecomposability of polygons, Section \ref{sec:dehn} extends the weight system from Section 3 of \cite{turnerwu} to develop three criteria which, if all satisfied, guarantee the existence of an equidecomposability relation. Conversely, if any of these criteria are not satisfied, there does not exist an equidecomposability relation between the polygons under consideration.

Our invariant for polygons consists of three pieces of data: morally speaking, one for each $d$-dimensional face, $0 \leq d \leq 2$, (vertices, edges, facets) of a simplicial decomposition of a polygon. The most puzzling criterion is the one corresponding to facets, which requires, roughly speaking, looking at the orbit of our polygon in a countable family of discrete dynamical systems. 

\begin{theorem}[Main Result 1]
\label{thm:main1}
We present new necessary and sufficient conditions for rational polygons $P$ and $Q$ to be rationally finitely discretely equidecomposable in terms of three pieces of data: (1) the Ehrhart quasi-polynomial, (2) a weighting system on the edges, and (3) a countable family of discrete dynamical systems generated by minimal triangulations. 
\end{theorem}

In Section \ref{sec:questions}, we conjecture that it suffices to only check one special member of this family of discrete dynamical systems to verify that the polygon satisfies the facet criterion. We summarize this with the following conjecture. 

\begin{conjecture}
Checking condition (3) from Theorem \ref{thm:main1} can be reduced to a finite number of verifications. In particular, there exists a finite procedure for determining if rational polygons $P$ and $Q$ are rationally discretely equidecomposable.  
\end{conjecture}

The next question posed by Haase--McAllister in \cite{haase}\footnote{Labeled Question 3.3 in \cite{haase}.} delves deeper into automating these verifications.  

\begin{question}
\label{que:alg}
Given that $P$ and $Q$ are equidecomposable rational polytopes, is there an algorithm to produce explicitly an equidecomposability relation $\mathcal{F}: P \to Q?$
\end{question}

In our special case of rational finite equidecomposability, the answer to the above question is ``yes'', although we do very little here to realize our algorithm computationally or analyze its time-complexity. Most likely the one we present is extremely computationally expensive. The upshot is that all of our methods in this paper are constructive. That is, if we know \textit{a priori} that $P$ and $Q$ are finitely rationally discretely equidecomposable, then there exists a procedure with a finite number of steps (albeit, potentially very large) for constructing an equidecomposability relation between $P$ and $Q$.

However, if handed two rational polygons $P$ and $Q$ and asked ``Are $P$ and $Q$ finitely rationally discretely equidecomposable?'', we do not have in this paper a finite algorithm for answering this question. This is because the ``facet-criterion'' mentioned under Question \ref{que:dehn} requires the user to check an infinite amount of data.

\begin{theorem}[Main Result 2]
There exists an algorithm for verifying that the criteria from Theorem \ref{thm:main1} are satisfied. Moreover, if (1),(2), and (3) from Theorem \ref{thm:main1} are satisfied, there exists an algorithm for explicitly constructing an equidecomposability relation between $P$ and $Q$. 
\end{theorem}

We close with the following remark that emphasizes the scope of our results.

\begin{remark}
Our results only concern the case of rational finite discrete equidecomposability of polygons. The questions posed above are still unanswered in higher dimensions or for irrational or infinite equidecomposability relations. See Sections 4 and 5 of \cite{turnerwu} for further discussion. 
\end{remark}

\end{subsection}

\begin{subsection}{Outline}

\begin{remark}
For the rest of this paper, we abbreviate the phrase \emph{rational finite discrete equidecomposability} with simply \emph{equidecomposability}, as all of our results concern this situation. Furthermore, we emphasize that we are only working with polytopes in dimension $2$ (polygons).
\end{remark}

\begin{itemize}

\item In Section \ref{sec:prelim}, we review the notation, definitions, and results from various sources that are needed for our work here.

\item In Section \ref{sec:dehn}, we present $3$ conditions required for two rational polygons $P$ and $Q$ to be equidecomposable. The first of these conditions is that the Ehrhart quasi-polynomials of $P$ and $Q$ agree. The second condition requires that a generalized version of the weights presented in Section 3 of \cite{turnerwu} that are assigned to $P$ and $Q$ must agree. The final condition is the most involved. For each denominator $d$, we construct a dynamical system $\mathcal{D}_d$ generated by the $d$-minimal triangles as defined in Section \ref{sec:prelim}. The third condition requires, roughly speaking, that the orbits of $P$ and $Q$ agree in this dynamical system for each $d$ a multiple of the denominators of $P$ and $Q$. 

\item Section \ref{sec:algorithm} realizes the criteria presented in Section \ref{sec:dehn} concretely. This highlights the extent to which our methods are constructive and provides foundations for further computational investigation. 

\item Section \ref{sec:questions} closes with questions for further research.
  
\end{itemize}

\end{subsection}

\end{section}

\begin{section}{Preliminaries}
\label{sec:prelim}

Here we recapitulate the notation, definitions, and results needed for this paper. 

\begin{definition}[Ehrhart function of a subset of $\mathbb{R}^n$]

Let $S$ be a bounded subset of $\mathbb{R}^n$. Then the \emph{Ehrhart function} of $S$ is defined to be

\begin{equation*}
\mathrm{ehr}_S(t) := |tS \cap \mathbb{Z}^n|
\end{equation*}

where $tS = \{ts \, | \, s \in S \}$ is the $t$'th dilate of $S$ \cite{ehrhart}. 

\end{definition}

A polytope $P$ is said to have denominator $d$ if $d$ is the least positive integer such that $dP$ has integer vertices (\textit{i.e.} $dP$ is an integral polytope). The following is Ehrhart's fundamental result \cite{ehrhart, beck}.

\begin{theorem}[Ehrhart]
Let $P$ be a rational polytope of denominator $d$. Then $\mathrm{ehr}_P(t)$ is a quasi-polynomial of period $d$. That is, 

\begin{displaymath}
   \mathrm{ehr}_P(t) = \left\{
     \begin{array}{lr}
       p_1(t) & : t \equiv 1 \mod d \\
       p_2(t) & : t \equiv 2 \mod d \\
       \hspace{1.3 cm} \vdots \\
       p_j(t) & : t \equiv j \mod d \\
       \hspace{1.3 cm} \vdots \\
       p_d(t) & : t \equiv d \mod d \\
     \end{array}
   \right.
\end{displaymath} 

where the $p_j$ are polynomials with degree equal to the dimension of $P$. 

\end{theorem}

We restrict our attention hereafter to (not necessarily convex) polygons in $\mathbb{R}^2$, as all of our results concern this case. Note that discrete equidecomposability has been studied in higher dimensions \cite{haase, stanley, kantor1}, but our results do not apply to that situation. 

The \emph{affine unimodular group} $G := GL_2(\mathbb{Z}) \rtimes \mathbb{Z}^2$ has the following action on $\mathbb{R}^2$ that preserves the integer lattice $\mathbb{Z} \times \mathbb{Z}$ (and, hence, Ehrhart functions as well). \\

\begin{center}
$x \mapsto gx := Ux + v$ 

$x \in \mathbb{R}^2, \, g = U \rtimes v \in G = GL_2(\mathbb{Z}) \rtimes \mathbb{Z}^2$. 
\end{center}

Note that $GL_2(\mathbb{Z})$ is precisely the set of integer $2 \times 2$ matrices with determinant $\pm 1$. Two regions $R_1$ and $R_2$ (usually polygons for our purposes) are said to be $G$\emph{-equivalent} if they lie in the same $G$-orbit, that is, $G R_1 = G R_2$. Also, a $G$\emph{-map} is a transformation on $\mathbb{R}^2$ induced by an element $g \in G$. We slightly abuse notation and refer to both the element and the map induced by the element as $g$.

In the same manner of \cite{haase}, we define the notion of discrete equidecomposability in $\mathbb{R}^2$.\footnote{The source \cite{haase} defines this notion in $\mathbb{R}^n$ and labels it $GL_n(\mathbb{Z})$\emph{-equidecomposability}. See definition 3.1.}

\begin{definition}[discrete equidecomposability]
\label{def:equi}

Let $P, Q \subset \mathbb{R}^2$. Then $P$ and $Q$ are \emph{discretely equidecomposable} if there exist open simplices $T_1, \ldots, T_r$ and $G$-maps $g_1, \ldots, g_r$ such that

\begin{equation*}
P = \bigsqcup _{i= 1} ^r T_i \, \, \mathrm{and} \, \, Q = \bigsqcup _{i = 1} ^r g_i(T_i).
\end{equation*}

Here, $\bigsqcup$ indicates disjoint union.

\end{definition}

For our purposes, $P$ ($Q$, respectively) is a polygon so we refer to the collection of open simplices $\{ T_1, \ldots, T_r \}$ ($\{ g_1(T_1), \ldots, g_r(T_r) \}$, respectively) as a \textit{simplicial decomposition} or \textit{triangulation}. 

If $P$ and $Q$ are discretely equidecomposable, then there exists a map $\mathcal{F}$ which we call the \emph{equidecomposability relation} that restricts to the specified $G$-map on each open piece of $P$. Precisely, that is, $\mathcal{F}|_{T_i} = g_i$. The map $\mathcal{F}$ is thus a piecewise $G$-bijection. Observe from the definition that the map $\mathcal{F}$ preserves the Ehrhart quasi-polynomial; hence $P$ and $Q$ are Ehrhart-equivalent if they are discretely equidecomposable.

The remainder of this section summarizes the discussion and some results from Section 2 of \cite{turnerwu}. Refer to that section for proofs of the statements below. 

\begin{definition}[minimal edge]
\label{def:minimal_edge}

Let $\mathcal{L}_d = \frac{1}{d} \mathbb{Z} \times \frac{1}{d} \mathbb{Z}$. A line segment $E$ with endpoints in $\mathcal{L}_d$ is said to be a \emph{$d$-minimal segment} if $E \cap \mathcal{L}_d$ consists precisely of the endpoints of $E$. 

\end{definition}

In Section 2 of \cite{turnerwu}, we defined a $G$-invariant weighting system on minimal edges. 

\begin{definition}[weight of an edge]
\label{def:weight}

Let $E_{p \to q}$ be an oriented $d$-minimal edge from endpoint $p = (\frac{w}{d}, \frac{x}{d})$ to $q = (\frac{y}{d}, \frac{z}{d})$. Then define the weight $W(E_{p \to q})$ of $E_{p \to q}$ to be

\begin{equation*}
W(E_{p \to q}) = \det \left[ \begin{pmatrix} d & 0 \\ 0 & d \end{pmatrix} \begin{pmatrix} w/d & y/d \\ x/d & z/d \end{pmatrix} \right] = \det \begin{pmatrix} w & y \\ x & z \end{pmatrix} \mod d.
\end{equation*}
\end{definition}

\begin{proposition}[$W$ is $G$-invariant]
\label{prop:weight_inv}

Let $E_{p \to q}$ be an oriented minimal edge and let $g \in G$. Then $W(E_{p \to q}) = \pm W(g(E)_{g(p) \to g(q)})$. Precisely, if $g$ is orientation preserving (i.e. $\det g=1$), $W(E_{p \to q}) = W(g(E)_{g(p) \to g(q)})$, and if $g$ is orientation reversing (i.e. $\det g=-1$), $W(E_{p \to q}) = -W(g(E)_{g(p) \to g(q)})$.

\end{proposition}

Minimal triangles are the $2$-dimensional version of minimal segments. 

\begin{definition}
A triangle $T$ is said to be $d$\emph{-minimal} if $T \cap \mathcal{L}_d$ consists precisely of the vertices of $T$. 
\end{definition}

Observe that the edges of a $d$-minimal triangle are $d$-minimal segments. Moreover, a $d$\textit{-minimal triangulation} is a triangulation whose facets are all $d$-minimal triangles.

In Section 3 of \cite{turnerwu}, we extended the weight on a minimal segment to a $G$-invariant weight on minimal triangles.

\begin{definition}
\label{def:weight_triangle}
Let $T$ be a $d$-minimal triangle with vertices $p, q,$ and $r$. Orient $T$ counterclockwise, and suppose WLOG this orients the edges $E^1, E^2, E^3$ of $T$ as follows: $E^1_{p \to q}, E^2_{q \to r},$ and $E^3_{r \to p}$. Then we define the weight $W(T)$ of $T$ to be the (unordered) multiset as follows:

\begin{equation}
W(T) := \left\{ W(E^1_{p \to q}), W(E^2_{q \to r}), W(E^3_{r \to p}) \right\}.
\label{eqn:weight_triangles}
\end{equation}

\end{definition}

\begin{proposition}[W(T) is invariant under $G$\footnote{In fact $W(T)$ is preserved by equidecomposability relations as well, but the proof of this is more difficult. See Section 3 of \cite{turnerwu}.}]
\label{prop:weight_invariant}
Let $T$ be a minimal triangle, and $g$ a $G$-map. Then

\begin{equation*}
W(T) = W(g(T)).
\end{equation*}
\end{proposition}

Our final preliminary result gives a one-to-one correspondence between weights and $G$-equivalence classes of $d$-minimal triangles. 

\begin{theorem}
\label{thm:weight_classification}
Two $d$-minimal triangles $S$ and $T$ are $G$-equivalent if and only if $W(S) = W(T)$. 
\end{theorem}

\end{section}

\begin{section}{A Necessary and Sufficient Condition for Equidecomposability}
\label{sec:dehn}

In this section, we will present a necessary and sufficient condition for rational polygons $P$ and $Q$ to be rationally discretely equidecomposable. This condition consists requires three pieces of data from a polygon --- heuristically speaking, one requirement for each of the possible types of simplices in a triangulation: \emph{vertices}, \emph{edges}, and \emph{facets}. After proving our results theoretically, we will realize them constructively in Section \ref{sec:algorithm} with an algorithm for detecting and constructing equidecomposability relations. 

First, we need the following viewpoint and convention. Any equidecomposability relation $\mathcal{F}:P \to Q$ between denominator $d$ polygons can be viewed as fixing a $d'$-minimal triangulation  $\mathcal{T}_1$ (for some denominator $d'$ divisible by $d$) of $P$ and assigning a $G$-map $g_F$ to each face $F$ (vertex, edge, or facet, respectively) of $\mathcal{T}_1$ such that $g(F)$ is a face (vertex, edge, or facet, respectively) of some minimal triangulation $\mathcal{T}_2$ of $Q$. Hence, when enumerating equidecomposability relations with domain $P$, it suffices to consider those $\mathcal{F}$ that assign $G$-maps to the faces of some minimal triangulation of $P$. Moreover, in this case we write $\mathcal{F}_{d'}: (P, \mathcal{T}_1) \to (Q, \mathcal{T}_2)$ to emphasize the underlying triangulations and their denominator.\footnote{See the beginning of Section 2 of \cite{turnerwu} for a more careful discussion of this set-up.} 

\begin{subsection}{Definitions and Notation}
\label{subsec:definitions}

First we introduce some notation. Given a triangulation $\mathcal{T}$ of some rational polygon, let $\mathcal{T}^{\dotr}$, $\mathcal{T}^{-}$, and $\mathcal{T}^{\Delta}$ denote the set of vertices, set of open edges, and set of open facets, respectively, of $\mathcal{T}$. Also, in consistency with this paper, all triangulations of polygons consist entirely of rational simplices.   

The following definition describes a sort of `partial' equidecomposability relation, a \emph{facet map}. 

\begin{definition}[facet map]

Let $P$ and $Q$ be rational polygons with triangulations $\mathcal{T}_1$ and $\mathcal{T}_2$, respectively. We say that a map  $\mathcal{F}^{\Delta}: \mathcal{T}_1^{\Delta} \to  \mathcal{T}_2^{\Delta}$ is a \emph{facet map} if the following criteria are satisfied. 

\begin{enumerate}
\item $\mathcal{F}^{\Delta}|_{\mathcal{T}_1^{\Delta}}$ is a bijection to $\mathcal{T}_2^{\Delta}$.
\item For all $F \in \mathcal{T}_1^{\Delta}$, $\mathcal{F}^{\Delta}|_{F}$ is a $G$-map. 
\end{enumerate}

\end{definition}

Note that according to this definition, the behavior of $\mathcal{F}^{\Delta}$ on open edges and vertices of $\mathcal{T}_1$ is disregarded --- all that is controlled is the map's behavior on the set of facets $\mathcal{T}_1^{\Delta}$. Also observe that any equidecomposability relation is automatically a facet map. Our first goal is to develop criteria describing when a facet map may be extended to an equidecomposability relation (see Proposition \ref{prop:extend}).  

Next, we present a generalized version $\mathbf{W}_d$ of the weight $W_d$ defined in Section \ref{sec:prelim}. First we will define $\mathbf{W}_d$ on $d$-minimal edges. Consider the set of edges of weight $\pm i$ with denominator $d$. In general, not all of these edges are $G$-equivalent.\footnote{Consider $3$-minimal segments of weight $0$. The segment $E_1$ from $(0,0) \to (1/3,0)$ and $E_2$ from $(1/3,0) \to (2/3,0)$ both have weight $0$, but they are not $G$-equivalent because $E_1$ contains an integer lattice point while $E_2$ does not.} Hence, number the different equivalence classes from $1$ to $N_{\pm i}$ (we will show there are a finite number of these in Section \ref{sec:algorithm}). If a $\pm i$ edge belongs to the equivalence class $j \in \{1, \ldots, N_{\pm i} \}$, we define

\begin{equation*}
\mathbf{W}_d(E) := (\pm i, j)
\end{equation*}

and say that $E$ is a Weight $(\pm i, j)$ $d$\emph{-minimal edge}. 

\begin{remark}
\label{rmk:edge_equiv}
Observe that $d$-minimal edges $E$ and $F$ are $G$-equivalent if and only if $\mathbf{W}_d(E) = \mathbf{W}_d(F)$. 
\end{remark}

Also note that the concept of Weight $\mathbf{W}$ differs from the weight $W$ because it does not take into account orientation. Also observe that we use the word ``Weight'', with a capitalized `w', to refer to $\mathbf{W}$ as opposed to the weight $W$. 

Now, given $d'$ divisible by $d$, we can define $\mathbf{W}_{d'}$ on an arbitrary denominator $d$ polygon $P$.

\begin{definition}
\label{def:Weight}
Let $P$ be a denominator $d$ polygon and $d'$ a positive integer divisible by $d$. Then the boundary of $P$ may be uniquely regarded as a union $\cup E^{i}$ of $d'$-minimal edges $\{E^i\}$. We define the Weight (using a capital letter to distinguish it from the weight $W_d$ described in Section \ref{sec:prelim}) $\mathbf{W}_{d'}(P)$ of the polygon $P$ to be the following unordered multiset

\begin{equation*}
\mathbf{W}_{d'} := \bigcup \left\{ \mathbf{W}_{d'}(E^i) \right\}.
\end{equation*}

\end{definition}

Note that we have not presented a computational method for determining the $G$-equivalence class $\mathbf{W}_{d}(E)$ of a $d$-minimal edge as we did for $W(E)$ in Section \ref{sec:prelim}. This will be delayed until Section \ref{sec:algorithm}, as none of our methods from Section \ref{subsec:dehn} require direct computation.

\end{subsection}

\begin{subsection}{Extending a Facet Map}
\label{subsec:dehn}

Our goal here is to prove the following proposition that describes when a facet map can be `extended' to an equidecomposability relation.

\begin{proposition}
\label{prop:extend}
Let $P$ and $Q$ be denominator $d$ polygons with triangulations $\mathcal{T}_1$ and $\mathcal{T}_2$, respectively. Suppose $\mathcal{F}^{\Delta}: \mathcal{T}_1^{\Delta} \to \mathcal{T}_2^{\Delta}$ is a facet map. Then there exists a rational discrete equidecomposability relation $\mathcal{F}: P \to Q$ if and only if the following two conditions are satisified: 

\begin{enumerate}
\item $\mathrm{ehr}_P = \mathrm{ehr}_Q$ \emph{[vertex compatibility]}, and
\item $\mathbf{W}_d(P) = \mathbf{W}_d(Q)$ \emph{[edge compatibility]}. 
\end{enumerate}
\end{proposition}

The given labeling of (1) and (2) is for the following reason. If [vertex compatibility] is satisfied, we will demonstrate that the map $\mathcal{F}^{\Delta}$ may be extended to send the vertices of $\mathcal{T}_1$ to the vertices of $\mathcal{T}_2$ via a piecewise $G$-bijection. Similarly, if [edge compatibility] holds, we show $\mathcal{F}^{\Delta}$ may be extended to send the edges of $\mathcal{T}_1$ to the edges of $\mathcal{T}_2$ via a piecewise $G$-bijection. To prove Proposition \ref{prop:extend}, we first deal with conditions (1) and (2) separately. 

\begin{subsubsection}{Vertex Compatibility}
\label{sub:vertex}

Our first goal is to show that Ehrhart equivalence of $P$ and $Q$ implies that the facet map $\mathcal{F}^{\Delta}$ can be extended to biject the set $\mathcal{T}_1^{\dotr}$ to $\mathcal{T}_2^{\dotr}$ with $G$-maps. We do so with the following set-up.

A point $p \in \mathbb{R}^2$ is said to be \emph{$d$-primitive} if $d$ is the least integer such that $p \in \mathcal{L}_d$. Equivalently, this implies that when the coordinates of $p$ are written in lowest terms, $d$ is the least common multiple of the denominators of each coordinate. Let $\mathcal{S}_d$ be the set of all $d$-primitive points. The next lemma shows that the $G$-orbit of a $d$-minimal point is exactly $\mathcal{S}_d$.

\begin{lemma}
\label{lem:points1}
Let $p \in \mathbb{R}^2$ be a $d$-primitive point. Then given $g \in G$, $g(p)$ is again a $d$-primitive point. Moreover given a $d$-primitive point $p' \in \mathcal{S}_d$, there exists a $G$-map sending $p$ to $p'$. 
\end{lemma}

Before proving this, it is necessary to have the following simple fact from elementary number theory. We provide a proof for the sake of completeness.

\begin{lemma}
\label{lem:nt}

Let $a,b,d \in \mathbb{Z}$. Suppose $\gcd(a,b)$ is relatively prime to $d$. Then there exists $q \in \mathbb{Z}$ such that $\gcd(a + qd, b) = 1$. 

\end{lemma}

\begin{proof}
Let $k = \gcd(a,b)$. By assumption, $\gcd(k,d) = 1$. Set $a = km$ and $b = kn$. Write $n$ as a product of (perhaps repeated) primes: $n = \prod_{i} s_i$. Then set $n' = \prod_{s_i | k} s_i$, and let $n = n'q$. Observe that $n'$ and $q$ share no common factors.  

Consider the number $a + qd = mk + qd$. We claim that $\gcd(a + qd, b) = 1$. To see this, first note that $b = kn'q$. By construction, we see $kn'$ and $q$ share no common factors. Suppose $r$ is a prime dividing $b$.  

If $r|kn'$, we show that $r$ does not divide $mk + qd$. First, $r|k$ because every factor of $n'$ is a factor of $k$ by construction. However, $r$ does not divide $d$ because $\gcd(k,d) = 1$. Finally, $r$ does not divide $q$ because $r|kn'$ and $kn'$ and $q$ have no common factors. 

If instead $r|q$, then to show $r$ does not divide $mk + qd$, it suffices to prove $r$ does not divide $mk$. First, $r$ does not divide $k$ because $k$ and $r$ share no common factors. Finally, $r$ does not divide $m$ because $r|q|n$ and $m$ and $n$ are relatively prime. With these two cases, we conclude that $a + qd$ and $b$ are relatively prime, as desired. 
\end{proof}

Now we return to the proof of Lemma \ref{lem:points1}.

\begin{proof}[Proof of Lemma \ref{lem:points1} ]
First we show that $d$-primitive points are $G$-mapped to $d$-primitive points. Since the lattice $\mathcal{L}_d$ is preserved by $G$, $g(p)$ is a point of $\mathcal{L}_d$ for all $g \in G$. We only need to check that $p$ cannot be sent to a $d'$ primitive point with $d' < d$. Suppose $q := g(p) \in \mathcal{L}_d'$ for some $g \in G$. Then it follows that $g^{-1}(q) = p$ is a point of $\mathcal{L}_{d'}$, since the lattice $\mathcal{L}_{d'}$ is also preserved by $G$. This contradicts the assumption that $p \in \mathcal{L}_d$. 

Now we show that all $d$-primitive points are $G$-equivalent. It is well-known that the $GL_2(\mathbb{Z})$-orbit of the point $p_0 = (1,0)$ consists of all \textit{visible points} in $\mathbb{Z}^2$, that is, all points of the form $(a,b)$ with $\gcd(a,b) = 1$. Therefore, the $GL_2(\mathbb{Z})$-orbit of $(\frac{1}{d}, 0)$ consists of all visible points in $\mathcal{L}_d$, that is, all points of the form $(\frac{m}{d}, \frac{n}{d})$ with $\gcd(m,n) = 1$. Therefore, we just need to show that the set of integer translations of the visible points in $\mathcal{L}_d$ is precisely the set of $d$-primitive points.

Let $p = (\frac{a}{d}, \frac{b}{d})$ be a $d$-primitive point. We will construct an integer vector $v$ such that $v + p$ is visible in $\mathcal{L}_d$. Let $k = \gcd(a,b)$. Note that $\gcd(k,d) = 1$, because otherwise, $p$ would not be $d$-primitive. Apply Lemma \ref{lem:nt} to produce $q$ such that $a + qd$ and $b$ are relatively prime.

Therefore, the translation $(q, 0) + (\frac{a}{d}, \frac{b}{d})$ is a visible point in $\mathcal{L}_d$, as desired. We conclude that all $d$-primitive points are $G$-equivalent. 
\end{proof}

The next lemma allows us to count the number of primitive points in a rational polygon in terms of its Ehrhart quasi-polynomial. 

\begin{lemma}
\label{lem:points2}
Given rational polygons $P$ and $Q$, $\mathrm{ehr}_{P}(t) = \mathrm{ehr}_{Q}(t)$ for all $t \in \mathbb{N}$ if and only if $|P \cap \mathcal{S}_N| = |Q \cap \mathcal{S}_N|$ for all $N \in \mathbb{N}$.
\end{lemma}

\begin{proof}
First we show the forward direction. As a base case, we see the following when $N = 1$. 
\begin{equation*}
|P \cap \mathcal{S}_1| = |P \cap \mathbb{Z}| = \mathrm{ehr}_{P}(1) = \mathrm{ehr}_{Q}(1) = |Q \cap \mathcal{S}_1| = |Q \cap \mathbb{Z}|
\end{equation*}

Now, let $N > 1$. To run induction, suppose $|\mathcal{S}_n \cap P| = |\mathcal{S}_n \cap Q|$ for all $n<N$. Observe that $\mathrm{ehr}_P(N) = |NP \cap \mathbb{Z}^2| = |P \cap \frac{1}{N} \mathbb{Z} \times \frac{1}{N} \mathbb{Z}| = |P \cap \mathcal{L}_N|$. Hence, we see that $|P \cap \mathcal{S}_N| = \mathrm{ehr}_{P}(N) - |(P \cap \mathcal{L}_N) - \mathcal{S}_N|$, where the second term on the right-hand side counts the number of non-primitive points in $\mathcal{L}_N$ contained in $P$. 

We can compute the number of such points using the fact that $\mathcal{L}_N = \sqcup_{n|N} \mathcal{S}_n$. 

\begin{equation*}
|(P \cap \mathcal{L}_N) - \mathcal{S}_N| = \sum _{n|N, n \neq N} |P \cap \mathcal{S}_n|
\end{equation*}

The same equation holds for $Q$ by the same reasoning.

\begin{equation*}
|(Q \cap \mathcal{L}_N) - \mathcal{S}_N| = \sum _{n|N, n \neq N} |Q \cap \mathcal{S}_n|
\end{equation*}

By the inductive assumption, 
\begin{equation*}
\sum _{n|N, n \neq N} |P \cap \mathcal{S}_n| = \sum _{n|N, n \neq N} |Q \cap \mathcal{S}_n|.
\end{equation*}

Hence $|P \cap \mathcal{S}_N| = |Q \cap \mathcal{S}_N|$, as desired.

The backward direction follows immediately from the fact that 

\begin{equation*}
\mathrm{ehr}_P(N) = \sum _{n|N} |P \cap \mathcal{S}_n|.  
\end{equation*}
\end{proof}

\begin{remark}
\label{rmk:vertex}
Observe that if $\mathrm{ehr}_{P}(t) = \mathrm{ehr}_{Q}(t)$, Lemmas \ref{lem:points1} and \ref{lem:points2} imply that, for all $d$, the vertices of any $d$-minimal triangulation of $P$ has a piecewise $G$-bijection to the vertices of any $d$-triangulation of $Q$. 
\end{remark}

This achieves our goal for extending the facet map to biject the set of vertices, condition (1) [vertex compatibility] from Proposition \ref{prop:extend}. 

\end{subsubsection}

\begin{subsubsection}{Edge Compatibility}

In this section, our task is to show that $\mathbf{W}_{d}(P) = \mathbf{W}_{d}(Q)$ implies that $\mathcal{F}^{\Delta}$ can be extended to biject $\mathcal{T}_1^{-}$ to $\mathcal{T}_2^{-}$ in such a way that given an edge $E \in \mathcal{T}_1^{-}$, $\mathcal{F}^{\Delta}|_{E}$ is a $G$-map. In other words, $\mathcal{F}^{\Delta}$ can be extended to be a piecewise $G$-bijection from $\mathcal{T}_1^{-}$ to $\mathcal{T}_2^{-}$. 

The first step requires a slight generalization of Equation 8 in Section 3 of \cite{turnerwu}. 

\begin{lemma}
\label{lem:Weight_eqn}

Suppose $\mathcal{T}_1$ is a triangulation of $P$ consisting of denominator $d'$ simplices ($d|d'$). Let $\mathbbm{1}_{(\pm i, j)}$, be the indicator function for the ordered pair $(\pm i, j)$ and set

\begin{equation*}
\Delta_n^{(\pm i, j)} (\mathcal{T}_1) = \left\{ F| \, F \in \mathcal{T}_1^{\Delta} \, \mathrm{and} \, \sum _{E \, \mathrm{an} \, \mathrm{edge} \, \mathrm{in} \, F} \mathbbm{1}_{(\pm i, j)}(\mathbf{W}_d(E)) = n  \right\}.
\end{equation*}

Then, taking the triangulation $\mathcal{T}_1$ to be implicit, we have

\begin{equation}
\begin{split}
\label{eqn:Weight}
\sum _{E \in \mathcal{T}_1^{-}} \mathbbm{1}_{(\pm i, j)}(\mathbf{W}_d(E)) = \frac{1}{2}\left(|\Delta_1^{(\pm i, j)}| + 2|\Delta_2^{(\pm i, j)}| + 3|\Delta_3^{(\pm i,j)}|\right) \\
+ \frac{1}{2} \sum _{\substack{E \, \in \mathcal{T}_1^{-} \\ E \subset \partial P}} \mathbbm{1}_{(\pm i,j)}(\mathbf{W}_d(E)). 
\end{split}
\end{equation}

\end{lemma} 

We omit the proof, as it is in exact analogy to the proof of Equation 8 from Lemma 3.13 of \cite{turnerwu}. 

Next, we recover the following uniqueness statement.

\begin{lemma}
\label{lem:weight_inherit}
If $d|d'$, then $\mathbf{W}_{d'}(P) = \mathbf{W}_{d'}(Q)$ if and only if $\mathbf{W}_{d}(P) = \mathbf{W}_{d}(Q)$.
\end{lemma}

\begin{proof}

As in the proof of Lemma \ref{lem:weight_inherit} we will show that $\mathbf{W}_{d'}(P)$ uniquely determines $\mathbf{W}_d(P)$. This would justify the forward statement of the lemma.

Set $n = d'/d$. Let $\{ E^k \}_{k = 1} ^N$ be the $d$-minimal segments on the boundary of $P$. The edge $E^k$ is the union of $n$ $d'$-minimal segments $\{ E^{k,\ell} \}_{\ell = 1} ^n$. We call these $E^{k,\ell}$ the \emph{child segments} of the \emph{parent segment} $E^k$. Also, observe that $\cup _{k = 1} ^N \{ E^{k,\ell} \}_{\ell = 1} ^n$ is the set of $d'$-minimal edges comprising the boundary of $P$. 

Let $(\pm i, j) \in \mathbf{W}_{d'}(P)$. Then there exists $k, \ell$ such that $\mathbf{W}_{d'}(E^{k, \ell}) = (\pm i, j)$. We claim that if $\mathbf{W}_{d'}(E^{k', \ell'}) = (\pm i, j)$, then $\mathbf{W}_{d}(E^{k'}) = \mathbf{W}_{d}(E^{k})$. In other words, the Weight of child segments uniquely determines the Weight of the parent segment. 

Recall from Remark \ref{rmk:edge_equiv} that $\mathbf{W}_{d'}(E^{k, \ell}) = \mathbf{W}_{d'}(E^{k', \ell'})$ if and only if there exists $g \in G$ such that $g(E^{k, \ell}) = E^{k', \ell'}$. Recall that for all positive integers $D$, $G$-maps preserve $D$-minimal segments. Therefore, $g(E^{k})$ is a $d$-minimal segment containing $E^{k', \ell'}$. Clearly, there is unique $d$-minimal segment containing $E^{k', \ell'}$, the segment $E^{k'}$. Therefore, we see $g(E^{k}) = E^{k'}$, so that, again by Remark \ref{rmk:edge_equiv}, $\mathbf{W}_{d}(E^{k'}) = \mathbf{W}_{d}(E^{k})$, as desired. 

By this reasoning, the different elements occuring in the multiset $\mathbf{W}_d(P)$ are uniquely determined by the elements in $\mathbf{W}_{d'}(P)$. To conclude, we just need to show that an element's multiplicity is uniquely determined by $\mathbf{W}_{d'}(P)$. 

Suppose that there are $m$ edges in $\cup _{k = 1} ^N \{ E^{k,\ell} \}_{\ell = 1} ^n$ that are children of a parent edge of Weight $(\pm i, j)$. Then since a parent $d$-minimal edge is divided into $n$ $d'$-minimal segments, we see there are precisely $m/n$ parent edges of Weight $(\pm i ,j)$. That is, the ordered pair $(\pm i, j)$ occurs exactly $m/n$ times in $\mathbf{W}_d(P)$.

This implies that $\mathbf{W}_{d}(P)$ is uniquely determined by $\mathbf{W}_{d'}(P)$. 

Finally, the backward statement is an immediate consequence of Remark \ref{rmk:edge_equiv}. 
\end{proof}

Now we can accomplish the goal of this section on edge compatiability.  

\begin{lemma}
\label{lem:edge} Let $P$ and $Q$ be denominator $d$ polygons with triangulations $\mathcal{T}_1$ and $\mathcal{T}_2$, respectively. Suppose $\mathcal{F}^{\Delta}: \mathcal{T}_1^{\Delta} \to  \mathcal{T}_2^{\Delta}$ is a facet map. Then $\mathbf{W}_d(P) = \mathbf{W}_d(Q)$ if and only if $\mathcal{F}^{\Delta}$ can be extended to restrict to a piecewise $G$-bijection from $\mathcal{T}_1^{-}$ to $\mathcal{T}_2^{-}$. That is,

\begin{enumerate}
\item $\mathcal{F}^{\Delta}|_{\mathcal{T}_1^{-}}$ is a bijection to $\mathcal{T}_2^{-}$, and

\item for all $E \in \mathcal{T}_1^{-}$, $\mathcal{F}^{\Delta}|_{E}$ is a $G$-map.
\end{enumerate}
\end{lemma}

\begin{proof}

First we prove the backwards direction. WLOG suppose $\mathcal{T}_1$ and $\mathcal{T}_2$ are $d'$-minimal triangles for some $d'$ divisible by $d$ (otherwise, we may refine the original triangulations to have the specified form). Suppose $\mathcal{F}^{\Delta}$ has properties (1) and (2). By Lemma \ref{lem:Weight_eqn}, 

\begin{equation}
\begin{split}
\label{eqn:Weight2}
\sum _{E \in \mathcal{T}_1^{-}} \mathbbm{1}_{(\pm i,j)}(\mathbf{W}_{d'}(E))
-  \frac{1}{2}\left(|\Delta_1^{(\pm i, j)}(\mathcal{T}_1)| + 2|\Delta_2^{(\pm i, j)}(\mathcal{T}_1)| + 3|\Delta_3^{(\pm i,j)}(\mathcal{T}_1)|\right) = \\ 
\frac{1}{2} \sum _{\substack{E \, \in \mathcal{T}_1^{-} \\ E \subset \partial P}} \mathbbm{1}_{(\pm i,j)}(\mathbf{W}_{d'}(E)). 
\end{split}
\end{equation}

Observe that the LHS of Equation \ref{eqn:Weight2} is invariant under the map $\mathcal{F}^{\Delta}$. Therefore, so is the RHS of Equation \ref{eqn:Weight2}. Moreover, the value of the RHS over all pairs $(\pm i, j)$ uniquely determines $\mathbf{W}_{d'}(P)$, and hence $\mathbf{W}_{d'}(Q)$ as well. By Lemma \ref{lem:weight_inherit}, this implies $\mathbf{W}_d(P) = \mathbf{W}_d(Q)$, as desired. 

For the other direction, let's suppose $\mathbf{W}_d(P) = \mathbf{W}_d(Q)$. Then using Equation \ref{eqn:Weight} and Lemma \ref{lem:weight_inherit}, we again recover Equation \ref{eqn:Weight2} above. Specifically, we see

\begin{equation}
\begin{split}
\label{eqn:WeightP}
\sum _{E \in \mathcal{T}_1^{-}} \mathbbm{1}_{(\pm i, j)}(\mathbf{W}_{d'}(E)) = \frac{1}{2}\left(|\Delta_1^{(\pm i, j)}(\mathcal{T}_1)| + 2|\Delta_2^{(\pm i, j)}(\mathcal{T}_1)| + 3|\Delta_3^{(\pm i,j)}(\mathcal{T}_1)|\right) \\
+ \frac{1}{2} \sum _{\substack{E \, \in \mathcal{T}_1^{-} \\ E \subset \partial P}} \mathbbm{1}_{(\pm i,j)}(\mathbf{W}_{d'}(E)). 
\end{split}
\end{equation}

Moreover, since $\mathcal{F}^{\Delta}$ is a piecewise $G$-bijection from $\mathcal{T}_1^{\Delta}$ to $\mathcal{T}_2^{\Delta}$

\begin{equation}
\begin{split}
\label{eqn:1.10.1}
|\Delta_1^{(\pm i, j)}(\mathcal{T}_1)| + 2|\Delta_2^{(\pm i, j)}(\mathcal{T}_1)| + 3|\Delta_3^{(\pm i,j)}(\mathcal{T}_1)| = \\
|\Delta_1^{(\pm i, j)}(\mathcal{T}_2)| + 2|\Delta_2^{(\pm i, j)}(\mathcal{T}_2)| + 3|\Delta_3^{(\pm i,j)}(\mathcal{T}_2)|.
\end{split}
\end{equation}

Also, since $\mathbf{W}_d(P) = \mathbf{W}_d(Q)$, we see $\mathbf{W}_{d'}(P) = \mathbf{W}_{d'}(Q)$ by Lemma \ref{lem:weight_inherit}. Thus,

\begin{equation}
\label{eqn:1.10.2}
\sum _{\substack{E \, \in \mathcal{T}_1^{-} \\ E \subset \partial P}} \mathbbm{1}_{(\pm i,j)}(\mathbf{W}_{d'}(E)) = \sum _{\substack{E \, \in \mathcal{T}_2^{-} \\ E \subset \partial Q}} \mathbbm{1}_{(\pm i,j)}(\mathbf{W}_{d'}(E)).
\end{equation}

Apply Equation \ref{eqn:Weight} to $(Q, \mathcal{T}_2)$ to see

\begin{equation}
\begin{split}
\label{eqn:WeightQ}
\sum _{E \in \mathcal{T}_2^{-}} \mathbbm{1}_{(\pm i, j)}(\mathbf{W}_{d'}(E)) = \frac{1}{2}\left(|\Delta_1^{(\pm i, j)}(\mathcal{T}_2)| + 2|\Delta_2^{(\pm i, j)}(\mathcal{T}_2)| + 3|\Delta_3^{(\pm i,j)}(\mathcal{T}_2)|\right) \\
+ \frac{1}{2} \sum _{\substack{E \, \in \mathcal{T}_2^{-} \\ E \subset \partial Q}} \mathbbm{1}_{(\pm i,j)}(\mathbf{W}_{d'}(E))
\end{split}
\end{equation}

Using Equations \ref{eqn:WeightP}, \ref{eqn:1.10.1}, \ref{eqn:1.10.2}, and \ref{eqn:WeightQ}, we see

\begin{equation}
\label{eqn:edge_count}
\sum _{E \in \mathcal{T}_1^{-}} \mathbbm{1}_{(\pm i, j)}(\mathbf{W}_{d'}(E)) = \sum _{E \in \mathcal{T}_2^{-}} \mathbbm{1}_{(\pm i, j)}(\mathbf{W}_{d'}(E)).
\end{equation}

Equation \ref{eqn:edge_count} implies that for every $(\pm i, j)$, $\mathcal{T}_1$ and $\mathcal{T}_2$ have the same amount of $(\pm i, j)$ $d'$-minimal edges. Since all $(\pm i, j)$ $d'$-minimal edges edges are $G$-equivalent, we can construct a piecewise $G$-bijection from $\mathcal{T}_1^{-}$ to $\mathcal{T}_2^{-}$ and extend $\mathcal{F}^{\Delta}$ to restrict to this bijection on $\mathcal{T}_1$. This completes the proof of the forward direction.

\end{proof}

\end{subsubsection}

Now we can complete the proof of Proposition \ref{prop:extend}.

\begin{proof}[Proof of Proposition \ref{prop:extend}]

Suppose first that $\mathcal{F}$ is a discrete equidecomposability relation from $P$ to $Q$. Recall that Ehrhart-quasipolynomials are preserved by equidecomposability relations. Furthermore, the backward direction of Lemma \ref{lem:edge} implies that $\mathbf{W}_d(P)$ is preserved by equidecomposability relations. This proves the forward direction of Proposition \ref{prop:extend}.   

Now suppose that (1)[vertex compatibility] and (2)[edge compatibility] hold from the statement of Proposition \ref{prop:extend}. Remark \ref{rmk:vertex} implies that we can construct a piecewise $G$-bijection $\phi_1$ from $\mathcal{T}_1^{\dotr}$ to $\mathcal{T}_2^{\dotr}$, and Lemma \ref{lem:edge} implies that we can construct a piecewise $G$-bijection $\phi_2$ from $\mathcal{T}_1^{-}$ to $\mathcal{T}_2^{-}$. Updating $\mathcal{F}^{\Delta}$ to restrict to $\phi_1$ on $\mathcal{T}_1^{\dotr}$ and to $\phi_2$ on $\mathcal{T}_1^{-}$ gives rise to an equidecomposability relation $\mathcal{F}$ between $P$ and $Q$, as desired.   

\end{proof}

\end{subsection}

\begin{subsection}{Existence of a Facet Map}
\label{subsec:dehn2}

Now, given denominator $d$ polygons $P$ and $Q$ with triangulations $\mathcal{T}_1$ and $\mathcal{T}_2$, respectively, we will come up with a criterion for the existence of a facet map $\mathcal{F}^{\Delta}:  \mathcal{T}_1^{\Delta} \to \mathcal{T}_2^{\Delta}$ in terms of a family of discrete dynamical systems $\mathcal{D}_{d'}$ (one for each $d'$ divisible by $d$), associated to $d'$-minimal triangulations of denominator $d$. This criterion combined with Proposition \ref{prop:extend} provides the main theorem of this section, Theorem \ref{thm:dehn}, a necessary and sufficient condition for rational discrete equidecomposabliity. 

\begin{definition}[weight class]
An unordered triple $\omega$ is said to be a $d$-\emph{weight class} if there exists a $d$-minimal triangle $S$ satisfying $W(S) = \omega$.  
\end{definition}

It is helpful to recall from Theorem \ref{thm:weight_classification} that $d$-weight classes are in bijection with $G$-equivalence classes of $d$-minimal triangles. 

Next we have an important concept: \emph{pseudo-flippability}.

\begin{definition}[pseudo-flippability]
\label{def:pseudoflip}
We say that $d$-weight classes $\omega_1$ and $\omega_2$ are \emph{pseudo-flippable} if there exist adjacent $d$-minimal triangles $S_1$ and $S_2$ in $\mathbb{R}^2$ with weights $\omega_1$ and $\omega_2$, respectively, that form a parallelogram.

\end{definition}

\begin{figure}[H]
    \centering
    \includegraphics[keepaspectratio=true, width=13cm]{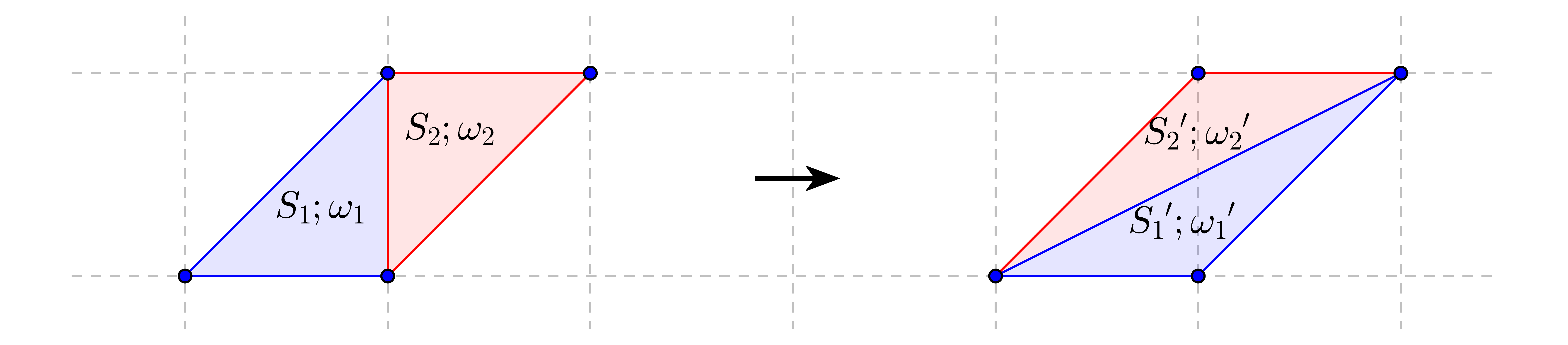}
    \caption{\small $S_1$ and $S_2$ before and after the flip. }
    \label{fig:pseudo-flip}
\end{figure}

This definition is motivated by the notions	``flippability'', ``flip moves'', and ``flip graphs'' from classical discrete geometry (see Chapter 3 of \cite{devadoss}). In this context, two triangles $S_1$ and $S_2$ in a triangulation are flippable if they are adjacent and form a convex quadrilateral. In this case, the diagonal of the quadrilateral that they form may be ``exchanged'' as in Figure \ref{fig:pseudo-flip}. This type of move is known in the literature as a \emph{flip}, which we denote by the notation $F_{S_1 S_2}$. Explicitly, given a triangulation $\mathcal{T}$ containing flippable triangles $S_1$ and $S_2$, $\mathcal{T}' = F_{S_1 S_2} \mathcal{T}$ is the new triangulation obtained by flipping $S_1$ and $S_2$. 

As a special case, note that flippable triangles in a $d$-minimal triangulation form a parallelogram, hence the reason for the seeming restrictiveness of Definition \ref{def:pseudoflip} (see the Introduction of \cite{caputo}).

The next task is to construct a discrete dynamical system $\mathcal{D}_{d'}$ that can detect certain facet maps, as we show in Proposition \ref{prop:facet_map_detect}. 

\begin{subsubsection}{Construction of $\mathcal{D}_{d'}(P)$}
\label{subsub:construct}
Let $P$ be a denominator $d$ polygon and let $d'$ be a positive integer divisible by $d$. Select an initial $d'$-minimal triangulation $\mathcal{T}_0(P)$ of $P$, and form the (unordered) multiset

\begin{equation}
\label{eqn:tau0}
\tau_0(P) := \bigcup _{F \in \mathcal{T}_0^{\Delta}(P)}\left\{ W_{d'}(F) \right\}
\end{equation}
of weights of $d'$-minimal triangles in $\mathcal{T}_0(P)$. We refer to this as the \emph{initial pseudo-triangulation}. Now we define an action called a \emph{pseudo-flip} $SF_{\omega_1 \omega_2}$ on $\tau_0(P)$, where $\omega_1$ and $\omega_2$ are pseudo-flippable, that will produce a new multiset $F_{\omega_1 \omega_2}\tau_0(P)$ of weights of $d'$-minimal triangles. We again refer to this multiset as a \emph{pseudo-triangulation}. 

Select triangles $S_1$ and $S_2$ with weights $\omega_1$ and $\omega_2$, respectively, that lie adjacent and form a parallelogram. The common side shared by $S_1$ and $S_2$ forms a diagonal. Exchanging this diagonal, as in Figure \ref{fig:pseudo-flip}, gives rise to two new $d'$-minimal triangles $S_1'$ and $S_2'$ with weights $\omega_1'$ and $\omega_2'$.\footnote{\label{foot:flip} From this construction, it is not clear that $\omega_1'$ and $\omega_2'$ do not depend on the initial choice of parallelogram. This turns out to be true, but to preserve continuity of this exposition, we will delay the proof of this fact until Section \ref{sec:algorithm} Proposition \ref{prop:well_defined}. } 

Thus, we define 
\begin{equation*}
SF_{\omega_1 \omega_2}\tau_0(P) = \left[\tau_0(P) \backslash \{ \omega_1, \omega_2 \}\right] \cup \{ \omega_1', \omega_2' \}.
\end{equation*}

\begin{definition}
\label{def:dds}
We define $\mathcal{D}_{d'}(P)$ to be the collection of all pseudo-triangulations that are obtainable from $\tau_0(P)$ by a series of pseudo-flips.
\end{definition}  

Observe that $\mathcal{D}_{d'}(P)$ is finite because there are a finite number of $G$-equivalence classes (hence a finite number of weight classes), and all pseudo-flip-equivalent pseudo-triangulations have the same finite cardinality. Moreover, $\mathcal{D}_{d'}(P)$ is independent of the initial choice of triangulation $\mathcal{T}_0(P)$. This is a consequence of a well-known theorem stating that any two $d'$-minimal triangulations $\mathcal{T}_0(P)$ and $\mathcal{T}_0'(P)$ are connected by a finite sequence of \emph{classical} flips, (see Chapter 3 of \cite{devadoss}). 

Explicity, there exists a sequence of flips $\{ F_{S_{1,i} S_{2,i}} \}_{i = 1} ^N$ such that 

\begin{equation*}
F_{S_{1,N} S_{2,N}} F_{S_{1,N-1} S_{2,N-1}} \cdots F_{S_{1,1} S_{2,1}} \mathcal{T}_0(P) = \mathcal{T}_0'(P). 
\end{equation*}

Observe from the definitions that any flip $F_{S_{1,j} S_{2,j}}$ induces a pseudo-flip $SF_{W_{d'}(S_{1,j}) W_{d'}(S_{2,j})}$. Let $W_{d'}(S_{i,j}) = \omega_{i,j}$. Therefore,

\begin{equation*}
SF_{\omega_{1,N} \omega_{2,N}} SF_{\omega_{1,N-1} \omega_{2,N-1}} \cdots SF_{\omega_{1,1} \omega_{2,1}} \tau_0(P) = \tau_0'(P).
\end{equation*}

Therefore, all possible initial pseudo-triangulations are connected under pseudo-flips, implying that Definition \ref{def:dds} is in fact well-defined. This concludes Section \ref{subsub:construct}.  

\end{subsubsection}

Now we return to the original purpose of Section \ref{subsec:dehn2}, a criterion for the existence of a facet map. 

\begin{proposition}[existence of a facet map]
\label{prop:facet_map_detect}
Let $P$ and $Q$ be denominator $d$ polygons. Then $\mathcal{D}_{d'}(P) = \mathcal{D}_{d'}(Q)$ for some $d'$ divisible by $d$ if and only if there exists a facet map $\mathcal{F}^{\Delta}: \mathcal{T}_1^{\Delta} \to \mathcal{T}_2^{\Delta}$ for some triangulation $\mathcal{T}_1$ of $P$ and some triangulation $\mathcal{T}_2$ of $Q$.  
\end{proposition}  

To prove this, it is helpful to first have the following lemma.

\begin{lemma}
\label{lem:dds_map}
Let $\mathcal{T}_0 = \mathcal{T}_0(P)$ be an initial $d'$-minimal triangulation of $P$ consisting of triangles $\{ S_i \}_{i= 1} ^N$ and let $\tau_0 = \tau_0(P) = \cup _{i = 1} ^N \{ \omega_i \}$ where $\omega_i = W_{d'}(S_i)$. Suppose $\tau \in \mathcal{D}_{d'}(P)$ where

\begin{equation}
\tau = \bigcup _{i = 1} ^N \left\{ \omega_i' \right\}
\end{equation} 
and that $\mathcal{C} = \sqcup _{i = 1} ^N S_i'$ is a disjoint collection of $d'$-minimal triangles satisfying $W_{d'}(S_i') = \omega_i'$. Then we may construct a facet map $\mathcal{F}^{\Delta}: \mathcal{T}_1^{\Delta} \to \mathcal{T}_2^{\Delta}$ for some triangulations $\mathcal{T}_1$ of $P$ and $\mathcal{T}_2$ of $\mathcal{C}$, respectively.
\end{lemma} 

\begin{proof}
We proceed by induction on the minimal number $n$ of pseudo-flips that it takes to transform $\tau_0$ into $\tau$. 

\emph{Base Case:} $n = 1$. Since we may relabel the weights and minimal triangles however we please, suppose WLOG that $SF_{\omega_1 \omega_2}\tau_0 = [\tau_0 \backslash \{ \omega_1, \omega_2 \}] \cup \{ \omega_1', \omega_2' \} = \tau$. Then clearly there exists a piecewise $G$-bijection from $\sqcup _{i = 3} ^N S_i$ to $\sqcup _{i = 3} ^N S_i'$ since both unions are disjoint and $\cup _{i = 3} ^N \{ \omega_i \}$ and $\cup _{i = 3} ^n \{\omega_i'\}$ are the same multiset.\footnote{Recall that if two $d'$-minimal triangles have the same weight, then they are $G$-equivalent. See Theorem \ref{thm:weight_classification}.}

Hence, we just need to construct a facet map from $S_1 \cup S_2$ to $S_1' \cup S_2'$. Since $W_{d'}(S_1)$ and $W_{d'}(S_2)$ are pseudo-flippable, there exist $G$-maps $g_1$ and $g_2$ so that $g_1(S_1)$ and $g_2(S_2)$ share a common edge $E_1$ and form a parallelogram $\mathcal{P}$. Construct the remaining diagonal $E_2$ of $\mathcal{P}$ besides the common edge $E_1$. This induces a triangulation $\mathcal{T}_{\mathcal{P}}$ on $\mathcal{P}$. If we say $g_1 \cup g_2$ is the piecewise $G$-map acting as $g_1$ on $S_1$ and as $g_2$ on $S_2$, respectively, then $\mathcal{T}_{S_1 \cup S_2} := (g_1 \cup g_2)^{-1}(\mathcal{T}_{\mathcal{P}})$ is a new triangulation on $S_1$ and $S_2$. Precisely, $\mathcal{T}_{S_1 \cup S_2}$ is the same as the union of the triangles $S_1$ and $S_2$, except we add a line segment bisecting an edge of $S_1$ and a line segment bisecting an edge of $S_2$. See Figure \ref{fig:S1_S2} for an illustration of this procedure.

By the well-definedness of pseudo-flippability (see Proposition \ref{prop:well_defined}), exchanging diagonal $E_1$ for $E_2$ produces two triangles of weight $W_{d'}(S_1')$ and $W_{d'}(S_2')$. Thus, there exist $G$-maps $g_1'$ and $g_2'$ such that $g_1'(S_1')$ and $g_2'(S_2')$ share the common edge $E_2$ and form the parallelogram $\mathcal{P}$. As in the previous passage, we get a triangulation $\mathcal{T}_{S_1' \cup S_2'} := (g_1' \cup g_2')^{-1}(\mathcal{T}_{\mathcal{P}})$ of $S_1' \cup S_2'$.  See Figure \ref{fig:S1'_S2'} for an illustration of this procedure.
\begin{figure}[H]
    \centering
    \includegraphics[keepaspectratio=true, width=17 cm]{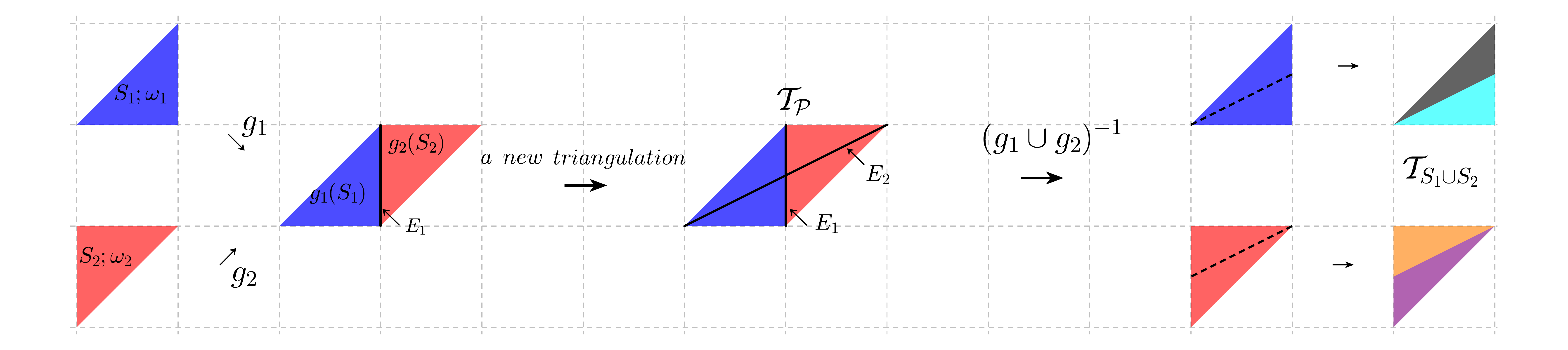}
    \caption{Construct a new triangulation on $S_1\cup S_2$. }
    \label{fig:S1_S2}
\end{figure} 
\begin{figure}[H]
    \centering
    \includegraphics[keepaspectratio=true, width=17 cm]{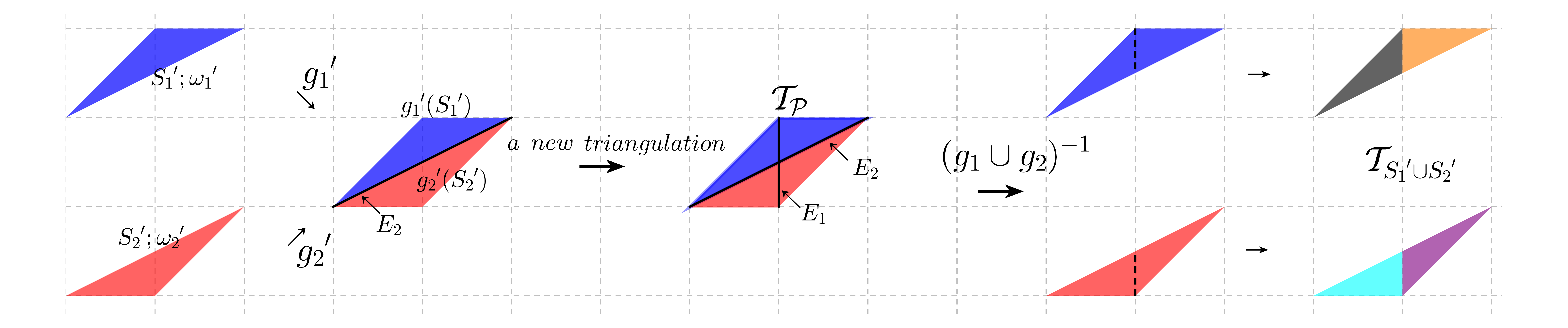}
    \caption{Construct a new triangulation on $S_1'\cup S_2'$. }
    \label{fig:S1'_S2'}
\end{figure} 

Then we have three facet maps in question: $g_1 \cup g_2: \mathcal{T}_{S_1 \cup S_2}^{\Delta} \to \mathcal{T}_{\mathcal{P}}^{\Delta}$, $\mathrm{id}: \mathcal{T}_{\mathcal{P}}^{\Delta} \to \mathcal{T}_{\mathcal{P}}^{\Delta}$, and $(g_1' \cup g_2')^{-1}:  \mathcal{T}_{\mathcal{P}}^{\Delta} \to  \mathcal{T}_{S_1' \cup S_2'}^{\Delta}$. Observe that the composition $(g_1 \cup g_2) \circ (g_1' \cup g_2')^{-1}:  \mathcal{T}_{S_1 \cup S_2}^{\Delta} \to \mathcal{T}_{S_1' \cup S_2'}^{\Delta}$ is a facet map from $S_1 \cup S_2$ to $S_1' \cup S_2'$. Combining this facet map with the piecewise $G$-bijection from $\cup _{i = 3} ^N S_i$ to $\cup _{i = 3} ^N S_i'$, we have a facet map from $P$ to $\mathcal{C}$, and the base case is complete. See Figure \ref{fig:map_btw_S}.\footnote{Note that in this construction, we completely ignore the troubles caused by edges $E_1$ and $E_2$, and only consider how faces are sent to one another. By the inverse map, we get a more refined triangulations on $S_1 \cup S_2$ and $S_1' \cup S_2'$, this enables us to send smaller pieces of faces from $S_1 \cup S_2$ to $S_1' \cup S_2'$ (In Figure \ref{fig:S1_S2},\ref{fig:S1'_S2'} and \ref{fig:map_btw_S}, the smaller pieces are drawn using different color to illustrate). The concept of refinement is the key to construct the facet map.}
\begin{figure}[H]
    \centering
    \includegraphics[keepaspectratio=true, width=17 cm]{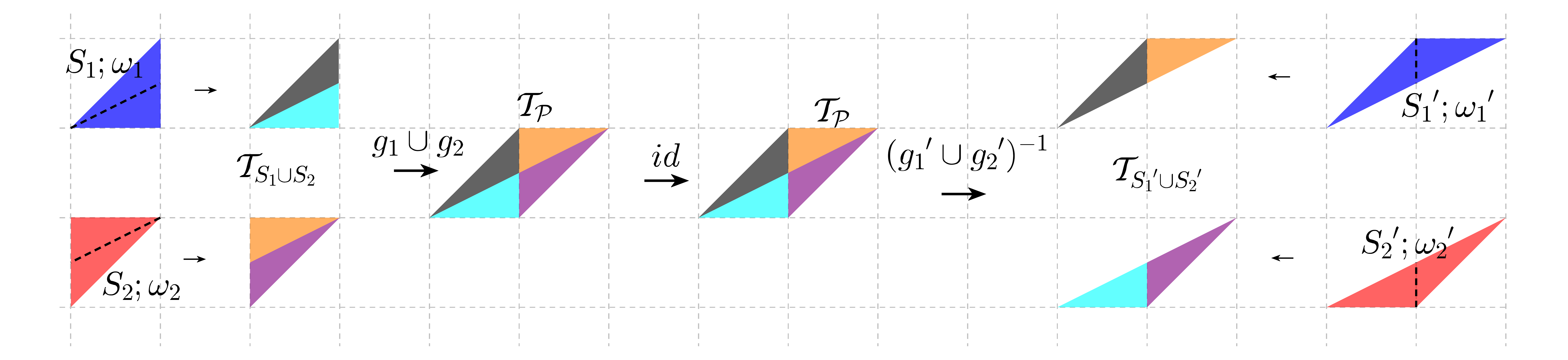}
    \caption{The facet map from $S_1\cup S_2$ to $S_1'\cup S_2'$. }
    \label{fig:map_btw_S}
\end{figure} 

\emph{Inductive step}: $n = k$. Suppose the lemma statement is true for all pseudo-triangulations less than $k$ pseudo-flips away from $\tau_0$, and that there is a path of $k$ pseudo-flips from $\tau_0$ to $\tau$. Suppose that $\tau' = \cup _{i = 1} ^N \{ \omega_i'' \}$ is the pseudo-triangulation immediately preceding $\tau$ on the selected path. 

Construct a disjoint set of $d'$ minimal triangles $S_i''$ with the property $W_{d'}(S_i'') = \omega_i''$. Let $S = \cup _{i = 1} ^N S_i''$ and $\tau' = \cup _{i = 1} ^N \{ \omega_i'' \}$. By our inductive assumption, there exists a facet map $\mathcal{F}_1^{\Delta}: \mathcal{T}_P^{\Delta} \to \mathcal{T}_S^{\Delta}$ for some triangulations $\mathcal{T}_P$ and $\mathcal{T}_S$. 

As in the base case, suppose $SF_{\omega_1'' \omega_2''}\tau' = [\tau' \backslash \{ \omega_1'', \omega_2'' \}] \cup \{ \omega_1', \omega_2' \} = \tau$. Repeating the procedure of the base case, we can construct a facet map $\mathcal{F}_2^{\Delta}$ from $S$ to $\mathcal{C}$ by adding two new line segments $E_1$ and $E_2$ to the triangulation $\mathcal{T}_S$ and applying a piecewise $G$-bijection to this new triangulation $\mathcal{T}_S'$. Formally, we write $\mathcal{F}_2^{\Delta}: \mathcal{T}_S'^{\Delta} \to \mathcal{T}_{\mathcal{C}}^{\Delta}$.

Next, add the pre-image of $E_1$ and $E_2$ under $\mathcal{F}_1^{\Delta}: \mathcal{T}_P^{\Delta} \to \mathcal{T}_S^{\Delta}$ to the triangulation $\mathcal{T}_P$ to get a new triangulation $\mathcal{T}_{P}'$ of $P$. One can then check that $\mathcal{F}_2^{\Delta} \mathcal{F}_1^{\Delta}: \mathcal{T}_P'^{\Delta} \to \mathcal{T}_{\mathcal{C}}^{\Delta}$ is a facet map, as desired.      
  
\end{proof}

\begin{proof}[Proof of Proposition \ref{prop:facet_map_detect}]
First we show the backwards direction. Suppose $\mathcal{F}^{\Delta}: \mathcal{T}_1^{\Delta} \to \mathcal{T}_2^{\Delta}$ is a facet map and that WLOG that $\mathcal{T}_1$ and $\mathcal{T}_2$ are $d'$-minimal triangulations (otherwise, we may refine them to be so). Then because $\mathcal{F}^{\Delta}$ is a piecewise $G$-bijection from $\mathcal{T}_1$ to $\mathcal{T}_2$, the multiset $\tau_0(P)$ of weights of the triangles in $\mathcal{T}_1$ must be the same as $\tau_0(Q)$. Then it is clear from Definition \ref{def:dds} that $\mathcal{D}_{d'}(P) = \mathcal{D}_{d'}(Q)$ --- select the initial triangulation of $P$ to be $\mathcal{T}_{1}$ and the initial triangulation of $\mathcal{C}$ to be $\mathcal{T}_2$ in the construction of $\mathcal{D}_{d'}$.   

For the forward direction, suppose $\mathcal{D}_{d'}(P) = \mathcal{D}_{d'}(Q)$. Select initial triangulation $\mathcal{T}_0$ of $P$ and $\mathcal{T}_0'$ of $Q$. Then $\tau_0(P)$ and $\tau_0(Q)$ are equivalent pseudo-triangulations under some sequence of pseudo-flips. Now apply Lemma \ref{lem:dds_map} with $\mathcal{C} = \mathcal{T}_0^{\Delta}$ and $\tau = \tau_0(Q)$ to get a facet map from $\mathcal{T}_1^{\Delta}$ to $\mathcal{T}_2^{\Delta}$ for some new choice of triangulations $\mathcal{T}_1$ of $P$ and $\mathcal{T}_2$ and $Q$. This concludes the proof of the proposition.  
\end{proof}
\end{subsection}

Finally, combining the results of Propositions \ref{prop:extend} and \ref{prop:facet_map_detect}, we obtain a necessary and sufficient condition for rational discrete equidecomposability. This is our response to Question \ref{que:dehn} posed by Haase--McAllister \cite{haase} in the case of rational discrete equidecomposability. 

\begin{theorem}[criteria for rational discrete equidecomposability]
\label{thm:dehn}
Let $P$ and $Q$ be denominator $d$ polygons. Then there exists a rational discrete equidecomposability relation $\mathcal{F}: (P, \mathcal{T}_1) \to (Q, \mathcal{T}_2)$ for some triangulation $\mathcal{T}_1$ of $P$ and $\mathcal{T}_2$ of $Q$ if and only if all of the following criteria are satisfied. 

\begin{enumerate}
\item $\mathrm{ehr}_P = \mathrm{ehr}_Q$ \emph{[vertex compatibility]}, and
\item $\mathbf{W}_d(P) = \mathbf{W}_d(Q)$ \emph{[edge compatibility]}
\item $\mathcal{D}_{d'}(P) = \mathcal{D}_{d'}(Q)$ for some $d'$ divisible by $d$ \emph{[}$d'$\emph{-facet compatibility]}
\end{enumerate}
\end{theorem}

\end{section}

\begin{section}{Computational Considerations and Algorithm for Equidecomposability}
\label{sec:algorithm}

The purpose of this section is to provide relatively concrete algorithms for detecting and constructing rational equidecomposability relations between denominator $d$ polygons $P$ and $Q$. The goal here is to demonstrate how our methods are constructive as well as to bring about some potentially interesting computational problems. To do so, we will analyze each component of Theorem \ref{thm:dehn} separately. 

\begin{subsection}{Detecting Facet Compatiblity and Mapping}
\label{subsec:facet_compat}

We begin by studying the weights of triangles produced by flips and pseudo-flips. This captures the perspective we take in realizing Section \ref{subsec:dehn2} more explicitly for the purposes of computation. It is again helpful to keep in mind Theorem \ref{thm:weight_classification}, which states that $d$-minimal triangles are $G$-equivalent if and only if they have the same weight. 

\begin{proposition}
\label{prop:pflip_rule}
Let $S$ and $T$ be $d$-minimal triangles sharing a common edge and forming a parallelogram $\mathcal{P}$. Suppose $W(S) = \{u, v, w\}$, $W(T) = \{x,y,z\}$\footnote{Recall from Definition \ref{def:weight_triangle} that these weights are computed by orienting each triangle counterclockwise.}, and $T$ and $S$ meet at the edge of $S$ with weight $u$ and the edge of $T$ with weight $x$. Let the edges with weight $v$ and $w$ be opposite to the edges with weight $y$ and $z$, respectively. Then,

\begin{enumerate}
\item $u \equiv -x \mod d$, 
\item $v + y \equiv 1 \mod d$, and
\item $w + z \equiv 1 \mod d$.
\end{enumerate}  

Moreover, when we exchange the diagonal of $\mathcal{P}$, we obtain two new minimal triangles $S'$ and $T'$ with weights $\{v, -w + 1, w - v\}$ and $\{w, -v + 1, v - w \}$, respectively, that meet at the edges labeled $\pm (v - w)$, and $v$ and $-w + 1$ are opposite to $-v + 1$ and $w$, respectively.  
\end{proposition}

\begin{figure}[H]
    \centering
    \includegraphics[keepaspectratio=true, width=13cm]{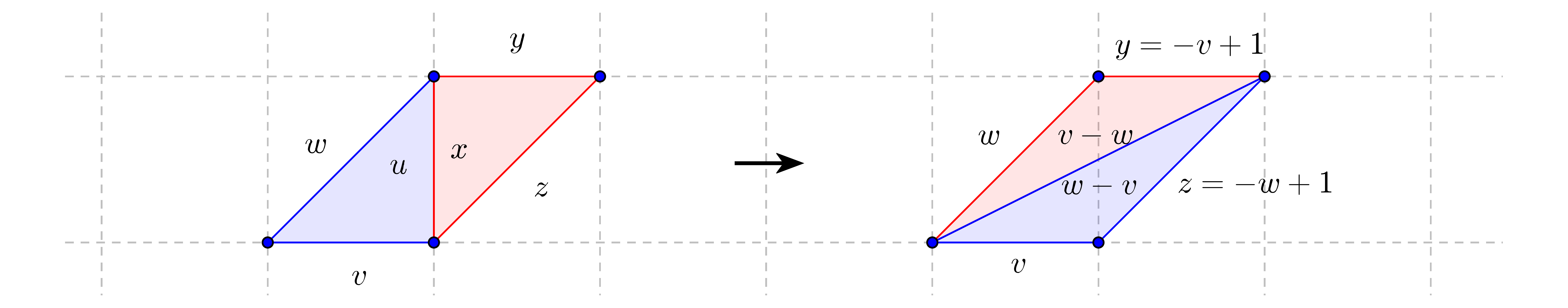}
    \caption{\small $S$ and $T$ before and after the flip. }
    \label{fig:flip_two_tri}
\end{figure}

\begin{proof}
We see that $u$ and $x$ must differ by a sign because they are the weight of the common edge between $S_1$ and $S_2$ when oriented in two different directions. 

Let $B$ be the edge of weight $v$ and $E$ the edge of weight $y$. By the geometry of the figure, opposite edges of $\mathcal{P}$ will be oriented in opposite directions. Moreover, the edge whose affine span (the line extending a given edge) lies further away from the origin will be oriented counterclockwise\footnote{See Definition 2.12 \cite{turnerwu} for precisely what we mean by \textit{counterclockwise oriented edge}.}. Suppose WLOG that the extended line of $B$ lies further from origin, \emph{i.e} $B$ is oriented counterclockwise and $E$ is oriented clockwise. Since $\mathcal{P}$ is a minimal parallelogram (in the sense that the only points at which it intersects $\mathcal{L}_d$ are its vertices), we observe using the definition of lattice distance (See Equation 2 in Section 2 of \cite{turnerwu}.) that $\mathrm{dis}(E) - \mathrm{dis}(B) = 1$. By the specified orientations of $B$ and $E$ and Proposition 2.13 from \cite{turnerwu}, this implies $y + v \equiv 1 \mod d$. The same procedure confirms property (3) as well. 

The final statement is a consequence of properties (1),(2), and (3) as well as the fact that the sum of the weights of a counterclockwise oriented $d$ minimal triangle is $1$ (See Proposition 2.15 from \cite{turnerwu} and Figure \ref{fig:flip_two_tri}). 
\end{proof}

Proposition \ref{prop:pflip_rule} comes with the following useful corollary.  

\begin{corollary}
\label{cor:pflip_cor}
The weights $\omega_1 = \{u', v', w'\}$ and $\omega_2 = \{x', y', z'\}$ are pseudo-flippable if and only if there exist bijections $\phi_1:\{u,v,w\} \to \{u',v',w'\}$ and $\phi_2:\{x,y,z\} \to \{x',y',z'\}$ such that 

\begin{enumerate}
\item $\phi_1(u) \equiv -\phi_2(x) \mod d$
\item $\phi_1(v) + \phi_2(y) \equiv 1 \mod d$, and
\item $\phi_1(w) + \phi_2(z) \equiv 1 \mod d$. 
\end{enumerate}
\end{corollary}

\begin{proof}
Suppose $\omega_1$ and $\omega_2$ are pseudo-flippable. Then there are triangles $S_1$ and $S_2$ with weights $\omega_1$ and $\omega_2$, respectively, forming a parallelogram $\mathcal{P}$. The weights of the edges of $\mathcal{P}$ and the diagonal along which $S_1$ and $S_2$ meet must satisfy the properties of Proposition \ref{prop:pflip_rule}. Hence, there is a labeling of these edges that will satisfy the properties of Proposition \ref{prop:pflip_rule}. This completes the forward direction.

For the backward direction, $G$-map $S_1$ to a translate $S_1'$ of $T_1 = \mathrm{Conv} \left((0,0), (\frac{1}{d}, 0), (0, \frac{1}{d}) \right)$ that lies in the unit square (Do so using Proposition 2.5 from \cite{turnerwu}). Consider a $d$-minimal right triangle $S_2'$ that borders the edge labeled $\phi(u)$ of $S_1'$.

\vspace{-.8 cm}

\begin{figure}[H]
    \centering
    \includegraphics[keepaspectratio=true, width=6cm]{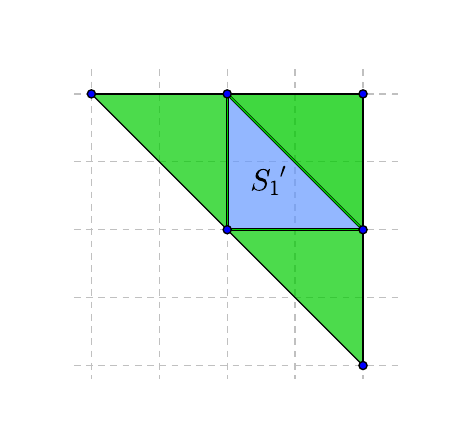}
    \caption{\small $S_1'$ is a translation of $T_1 = \mathrm{Conv} \left((0,0), (\frac{1}{d}, 0), (0, \frac{1}{d}) \right)$. The green triangles are potential choices of $S_2'$, the right triangle that borders $S_1'$ at a side of weight $\phi(u)$. Note that $S_1'$ and $S_2'$ form a minimal parallelogram.}
    \label{fig:S2prime}
\end{figure}

Apply Proposition \ref{prop:pflip_rule} and our assumption about the weight of $S_2$ to see that $W(S_1) = W(S_2)$. Therefore, $S_2$ can be mapped to $S_2'$ using Theorem \ref{thm:weight_classification}. By definition, $\omega_1$ and $\omega_2$ are pseudo-flippable.

\end{proof}

Thus, it is easy to check using Corollary \ref{cor:pflip_cor} that a given weight class is pseudo-flippable with at most three other weight classes. 

Now we can tie up a loose end: namely the content of Footnote \ref{foot:flip} from Section \ref{subsec:dehn2} stating that pseudo-flipping is a well-defined operation. 

\begin{proposition}[well-definedness of pseudo-flips]
\label{prop:well_defined}
Pseudo-flipping is well-defined. That is, given any two flippable (in the classical sense) pairs of $d$-minimal triangles $(S_1, S_2)$ and $(S_1', S_2')$ satisfying $W(S_1) = W(S_1') = \omega_1$ and $W(S_2) = W(S_2') = \omega_2$, exchanging the diagonals of the parallelograms $\mathcal{P}$ and $\mathcal{P'}$ from the common edge shared by $(S_1, S_2)$ and $(S_1', S_2')$, respectively, results in minimal triangles $(S_3, S_4)$ and $(S_3', S_4')$, respectively, satisfying $W(S_3) = W(S_3') = \omega_3$ and $W(S_4) = W(S_4') = \omega_4$. 
\end{proposition}

\begin{proof}
This proposition is trivial when $d = 1$, for then all $d$-minimal triangles are $G$-equivalent. Thus suppose $d > 1$. 

Using Proposition \ref{prop:pflip_rule}, write WLOG $W(S_1) = \omega_1 = \{u, v, w\}$ and $W(S_2) = \omega_2 = \{-u, -v + 1, -w - 1 \}$. That is, $S_1$ and $S_2$ border along the edge of weight $\pm u$, and the edge of weight $v$ ($w$, respectively) is opposite to the edge of weight $-v+1$ ($-w + 1$, respectively). It is a consequence of Proposition \ref{prop:pflip_rule} that for any parallelogram $\mathcal{P}'$ formed by triangles of weight $\omega_1$ and $\omega_2$ that share a common edge of weight $\pm u$, the new weights produced from exchanging diagonals will be $\omega_3 = \{v, -w + 1, w - v\}$ and $\omega_4 = \{w, -v + 1, v - w \}$.

Therefore, the only way well-definedness can break is if some $d$-minimal triangles $S_1'$ and $S_2'$ can meet along edges other than the edge of weight $\pm u$ and still form a parallelogram. Observe that the edge labeled $v$ ($w$, respectively) cannot meet the edge of weight $- v + 1$ ($-w + 1$, respectively) because $v \neq v - 1 \mod d$ ($w \neq w - 1 \mod d$, respectively). This takes care of our first two cases immediately. 

Recall using Proposition 2.15 from \cite{turnerwu} that $u = -w - v + 1 \mod d$. We have six remaining cases to consider.

\textbf{3. The edge} $\mathbf{w}$ \textbf{meets the edge} $\mathbf{-v + 1}$\textbf{, and the triangles form a parallelogram:} It is helpful to follow this argument along with Figure \ref{fig:flip3}. The assumption of this case implies $w \equiv v - 1 \mod d$. Therefore, $W(S_1') = \{u, v, w\} = \{-2w, w + 1, w \}$ and $W(S_2') = \{-u, -v + 1, -w + 1\} = \{2w, -w, -w + 1 \}$. Then we have two cases two consider: the first being that the edges of $S_1'$ labeled $-2w$ and $w + 1$ are opposite to the edges labeled $2w$ and $-w + 1$, respectively. If these form a parallelogram, then opposite labels would have to add up to $1 \mod d$, but since $2w + (-2w) \equiv 0 \mod d$, this is not the case.

\begin{figure}[ht]
    \centering
    \includegraphics[keepaspectratio=true, width=14cm]{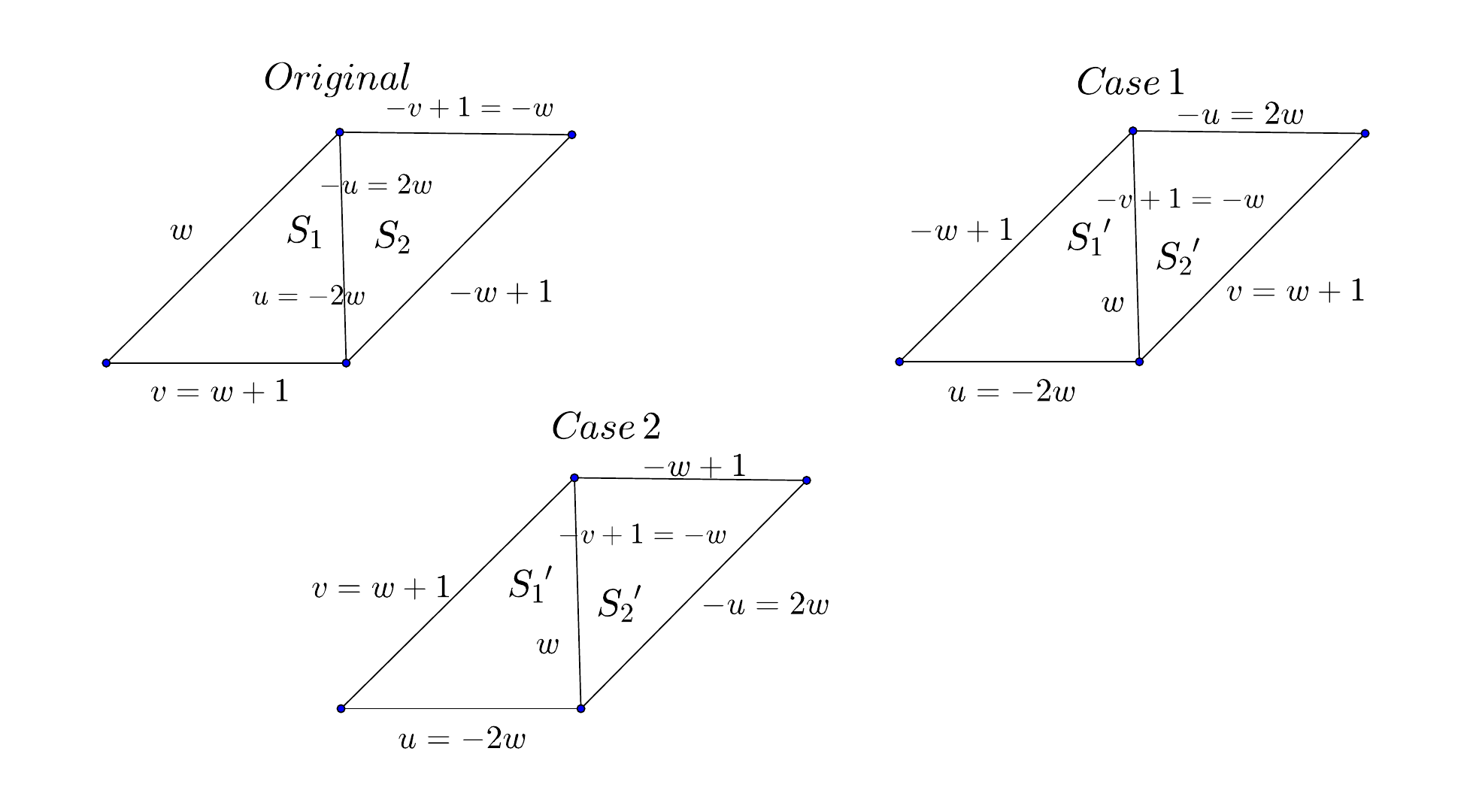}
    \caption{\small The figure depicts the original case how $S_1$ and $S_2$ meet. Under the assumption of \textbf{3.} we have two cases illustrating the possible ways that triangles $S_1'$ and $S_2'$ could meet. The first case turns out to be impossible and the second is redundant.}
    \label{fig:flip3}
\end{figure}

In the other case, we have the edges labeled $-2w$ and $w +1$ of $S_1'$ are opposite to edges $-w + 1$ and $2w$, respectively, of $S_2'$. If this is a parallelogram, then $2w + w + 1 \equiv 3w + 1 \equiv 1 \mod d$, which implies $2w \equiv -w \mod d$. Therefore, $W(S_1') = \{-2w, w, w +1 \}$ and $W(S_2') = \{2w, -w + 1, -w\}$, where the triangles $S_1'$ and $S_2'$ meet at the edge labeled $\pm 2w$, and the edges labeled $w$ and $w + 1$ are opposite to the edges labeled $-w + 1$ and $w$, respectively. But this is the same combinatorial set-up as the way triangles $S_1$ and $S_2$ meet. Thus, the flips of $(S_1, S_2)$ and $(S_1', S_2')$ agree. 

\textbf{4. The edge labeled} $\mathbf{v}$ \textbf{meets the edge labeled} $\mathbf{-w + 1}$\textbf{:} This case is subsumed by Case 3 by symmetry. 

\textbf{5. The edge labeled} $\mathbf{u}$ \textbf{meets the edge labeled} $\mathbf{-w + 1}$\textbf{:} In this case $-u \equiv -w + 1$, so that $W(S_2) = \{-u, -v + 1, -w + 1 \} = \{-w + 1, -v + 1,  -w + 1,\}$. Thus the only potentially combinatorially different possibility from the original pairing of $S_1$ and $S_2$ is if $S_1'$ and $S_2'$ meet along an labeled $-w + 1$, and the edges $w$ and $v$ of $S_1'$ are opposite to the edges $-v + 1$ and $-w + 1$ of $S_2$'. By the rule for minimal parallelograms, $w - v + 1 \equiv 1 \mod d$, which implies $w \equiv v \mod d$. But it that case, $W(S_2') = \{-w + 1, -w + 1, -w + 1 \}$ --- all of the edges of $S_2'$ have the same weight. Thus, the specified pairing of $S_1'$ and $S_2'$ is no different combinatorially from the original pairing of $S_1$ and $S_2$. That concludes the argument for Case 5. 

\textbf{6. The edge labeled} $\mathbf{u}$ \textbf{meets the edge labeled} $\mathbf{-v + 1}$\textbf{:} This case is subsumed by Case 5 by symmetry. 

\textbf{7. The edge labeled} $\mathbf{-u}$ \textbf{meets the edge labeled} $\mathbf{v}$\textbf{:} This case is subsumed by Case 5 by symmetry.

\textbf{8. The edge labeled} $\mathbf{-u}$ \textbf{meets the edge labeled} $\mathbf{w}$\textbf{:} This case is subsumed by Case 5 by symmetry.

Since edges can be paired up in $9$ ways (including the original pairing as $S_1$ and $S_2$ meet), this handles all necessary cases. Pseudo-flipping is a well-defined operation. 
  
\end{proof}

Let $P$ be denominator $d$ polygons and fix $d'$ divisible by $d$. Then Corollary \ref{cor:pflip_cor} and Proposition \ref{prop:well_defined} allows us to construct an algorithm for computing $\mathcal{D}_{d'}(P)$ and hence detect facet maps between denominator $d$ polygons $P$ and $Q$.

\begin{algorithm}[Computing $\mathcal{D}_{d'}$]
\label{alg:facet_map}
\normalfont
Let $P$ be a denominator $d$ polygon. The following brute-force method allows us to compute $\mathcal{D}_{d'}(P)$. \\
 
\begin{enumerate}
\item Find an initial $d'$-minimal triangulation $\mathcal{T}_0 = \mathcal{T}_0(P)$.
\item Use Definition \ref{def:weight_triangle} to compute the initial psuedo-triangulation $\tau = \tau_0(P)$.
\item Use Corollary \ref{cor:pflip_cor} to determine the set $S$ of pseudo-flippable pairs in $\tau$.
\item Using Proposition \ref{prop:pflip_rule}, compute each new pseudo-triangulation $\tau_1, \ldots, \tau_k$ arising from pseudo-flipping a pair in $S$. 
\item If any $\tau_i$ has not yet seen before, repeat steps (2) - (5) with $\tau = \tau_i$. If there are no new $\tau_i$, STOP.
\end{enumerate}
\end{algorithm}

For fixed $d'$, this process must terminate eventually because $\mathcal{D}_{d'}(P)$ is finite. However, to detect a facet map between $P$ and $Q$, we would have to execute Algorithm \ref{alg:facet_map} for each $d'$ divisible by $d$. 

\begin{observation}
\label{obs:facet_map}
\normalfont
Suppose we know $\mathcal{D}_{d'}(P) = \mathcal{D}_{d'}(Q)$. Using the methods of Lemma \ref{lem:dds_map}, we can also construct an algorithm for producing a facet map between $P$ and $Q$. First, select initial triangulations $\mathcal{T}_0(P)$ and $\mathcal{T}_0'(Q)$. Find a path of pseudo-flips from $\tau_0(P)$ to $\tau_0(Q)$ using some modified version of Algorithm \ref{alg:facet_map} where we also save path information. Then the inductive procedure outlined in the proof of Lemma \ref{lem:dds_map} can be used to systematically update the triangulation on $P$ and intermediary facet maps $\mathcal{F}_i$ at each step along the path of pseudo-triangulations. For the sake of brevity, we neglect presenting complete details. 
\end{observation}

Fix a positive integer $d$. In light of Proposition \ref{prop:well_defined}, it makes sense to construct the following graph, $\mathbb{G}_{d}$, that provides an explicit visualization of the dynamics arising from $\mathcal{D}_{d'}$. 

\begin{definition}[$\mathbb{G}_d$]

Given a positive integer $d$, the edge-labeled graph $\mathbb{G}_d$ is defined as follows.

\begin{enumerate}
\item \emph{[vertices]} $d$-weight classes.\footnote{Or equivalently, $G$-equivalence classes of $d$-minimal triangles by Theorem \ref{thm:weight_classification}.}
\item \emph{[edges]} weight classes $\omega_1$ and $\omega_2$ are connected by an edge if they are pseudo-flippable. 
\item \emph{[edge-labels]} The edge $E$ between $\omega_1$ and $\omega_2$ is labeled by the pair $\{\omega_1', \omega_2'\}$, the result of pseudo-flipping $\omega_1$ and $\omega_2$. 
\end{enumerate}

\end{definition} 

It is possible to write down explicitly the form of $\mathbb{G}_d$ for general $d$, although we will not present this here. Consider the weight class labeling of $6$-minimal triangles given below. By Section 2 of \cite{turnerwu}, this is a complete set of $G$-equivalence classes in denominator $6$.

\begin{figure}[H]
    \centering
    \includegraphics[keepaspectratio=true, width=10 cm]{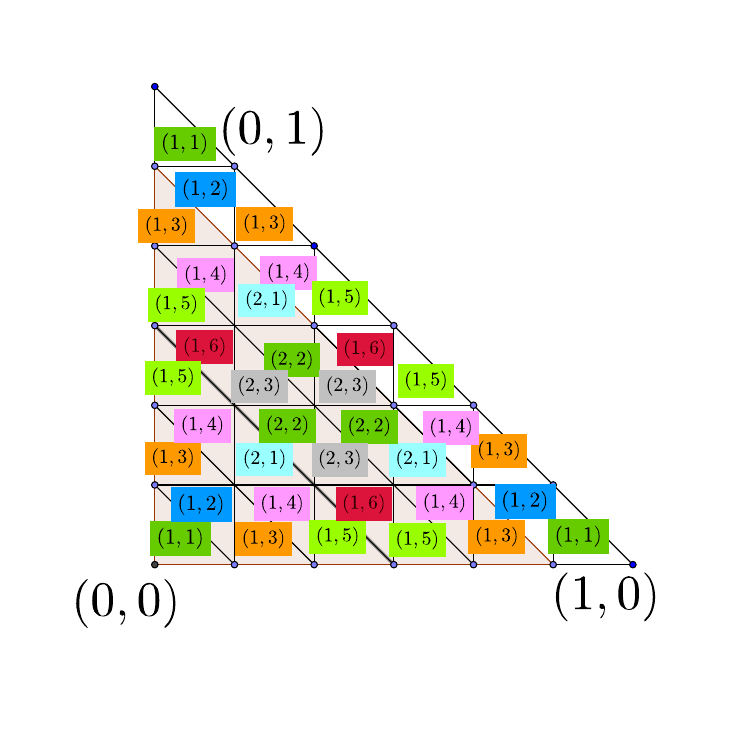}
    \label{fig:triang}
\end{figure}

\vspace{-1 cm}

Then $\mathbb{G}_6$ has the following form.

\begin{figure}[H]
    \centering
    \includegraphics[keepaspectratio=true, width=15 cm]{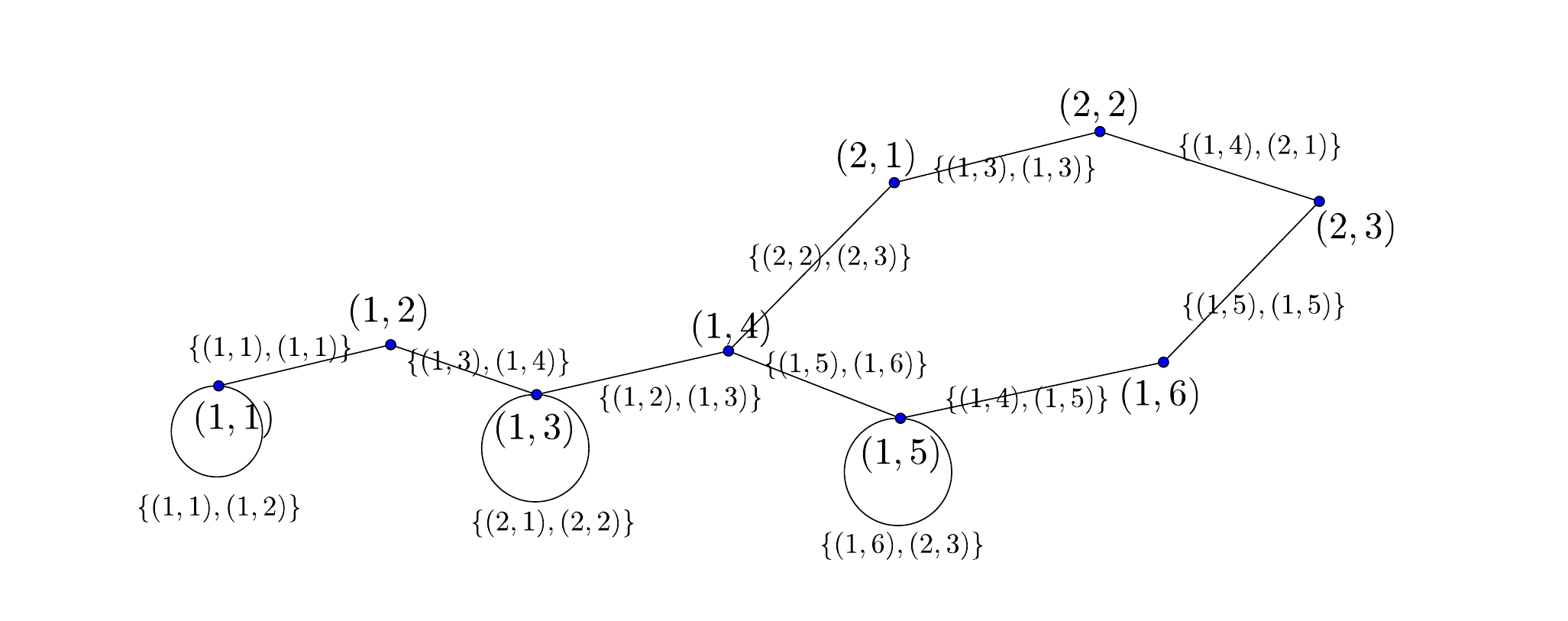}
    \caption{The graph $\mathbb{G}_d$ when $d = 6$. }
    \label{fig:Gd}
\end{figure} 

Observe that $\mathbb{G}_6$ is a subgraph of the dual to the triangular lattice, except for the loops on the lower row. That is, except for loops, it is a subgraph of the hexagonal lattice. This is true for general $d$ as well. 

Finally, observe that $\mathcal{D}_{d}(P)$ can be viewed as the orbit of a multiset of vertices of $\mathbb{G}_d$ under the pseudo-flip dynamics described by the edge-labelings. This alternative perspective on the discrete dynamical system $\mathcal{D}_{d'}$ provides explicit visualizations that may be fruitful for future inquiry.
\end{subsection}

\begin{subsection}{Detecting Edge Compatibility and Mapping}
\label{subsec:edge_compat}

In this section, we provide explicit methods for showing that two minimal edges are $G$-equivalent. In other words, given $d$-minimal edges $E$ and $F$, we provide a computational check to determine if $\mathbf{W}_d(E) = \mathbf{W}_d(F)$. By extension, given denominator $d$ polygons $P$ and $Q$, this can be used to check if $\mathbf{W}_d(P) = \mathbf{W}_d(Q)$. 

An oriented segment $E$ is said to be $d$-primitive if it is $d$-minimal and has the property that $W_d(E)$ is a unit in $\mathbb{Z}/d \mathbb{Z}$. 

\begin{lemma}
\label{lem:prim_segment}
If $E$ and $F$ are oriented $d$-primitive segments satisfying $W_{d}(E) = \pm W_{d}(F)$, then $E$ and $F$ are $G$-equivalent. Moreover, if $p$ is an endpoint of $E$ and $p'$ and endpoint of $F$, a $G$-equivalence may be chosen that sends $p$ to $p'$. 
\end{lemma} 

\begin{proof}
Suppose $E$ is oriented from endpoints $p = (w/d, x/d)$ to $q = (y/d, z/d)$, and $F$ is oriented from endpoints $p' = (w'/d, x'/d)$ to $q' = (y'/d, z'/d)$. Then observe that the existence of a $G$-map follows from the existence of a matrix $M = \begin{pmatrix} m_{11} & m_{12} \\ m_{21} & m_{22} \end{pmatrix} \in GL_2(\mathbb{Z})$ satisfying

\begin{equation*}
\begin{split}
ME = F \Leftrightarrow \\
\begin{pmatrix} m_{11} & m_{12} \\ m_{21} & m_{22} \end{pmatrix} \begin{pmatrix} w & y \\ x & z \end{pmatrix} = \begin{pmatrix} w' & y' \\ x' & z' \end{pmatrix} \mod d. 
\end{split}
\end{equation*}

Since $W(E) = \det \begin{pmatrix} w & y \\ x & z \end{pmatrix}$ is a unit in $\mathbb{Z}/d\mathbb{Z}$, $\begin{pmatrix} w & y \\ x & z \end{pmatrix}$ is invertible. Therefore,

\begin{equation*}
\begin{pmatrix} m_{11} & m_{12} \\ m_{21} & m_{22} \end{pmatrix} =  \begin{pmatrix} w' & y' \\ x' & z' \end{pmatrix} \begin{pmatrix} w & y \\ x & z \end{pmatrix}^{-1} \mod d. 
\end{equation*}

Observe by the fact that determinant is a multiplicative homomorphism that 

\begin{equation*}
\begin{split}
\det M = \left(\det \begin{pmatrix} w' & y' \\ x' & z' \end{pmatrix}\right) \left( \det \begin{pmatrix} w & y \\ x & z \end{pmatrix}\right)^{-1} \mod d \Leftrightarrow \\
\det M = W_{d}(F) W_{d}(E)^{-1} = \pm 1 \mod d. 
\end{split}
\end{equation*}

This implies that $\det M = \det \begin{pmatrix} m_{11} & m_{12} \\ m_{21} & m_{22} \end{pmatrix}  = \pm 1 + kd$ for some integer $k$. We need to construct a new matrix $M'$ such that 
\begin{enumerate}
\item $M = M' \mod d$, and 
\item $\det M' = \pm 1$. 
\end{enumerate}

to guarantee that the residue class of $M$ can be realized as a matrix $M'$ in $GL_2(\mathbb{Z})$.

\emph{Step 1.}

First we replace $m_{21}$ with a new entry $m_{21}' = m_{21} + n_{21} d$ such that $\mathrm{gcd}(m_{21} + n_{21} d, m_{22}) = 1$. Note that $\det M = \pm 1 \mod d$ implies $\mathrm{gcd}(m_{21}, m_{22})$ is relatively prime to $d$. Apply Lemma \ref{lem:nt} with $a = m_{21}, b = m_{22}$ and $d = d$ to produce $q = n_{21}$ such that $m_{21}' = m_{21} + n_{21} d$ is relatively prime to $m_{22}$. This concludes Step 1. 

\emph{Step 2:}

Next, we want to construct $n_{11}$ and $n_{12}$ such that the matrix

\begin{equation*}
M' := \begin{pmatrix} m_{11} + n_{11} d & m_{12} + n_{12} d \\ m_{21}' & m_{22} \end{pmatrix}.
\end{equation*}

has determinant equal to $\pm 1$. Observe that $\det M \equiv \det M' \equiv \pm 1 \mod d$, which implies $\det M = m_{11} m_{22} - m_{12} m_{21}' = \pm 1 + kd$ for some integer $k$.

We compute that 

\begin{equation*}
\det M' = m_{11} m_{22} - m_{12} m_{21}' + n_{11} d m_{22} - n_{12} d m_{21}'.
\end{equation*}

We wish to solve for $n_{12}$ and $n_{11}$ such that

\begin{equation}
\begin{split}
\label{eqn:euclid}
m_{11} m_{22} - m_{12} m_{21}' + n_{11} d m_{22} - n_{12} d m_{21}' = \pm 1 \Leftrightarrow \\
d( n_{11} m_{22} - n_{12} m_{21}') = \pm 1 - (m_{11} m_{22} - m_{12} m_{21}) \Leftrightarrow \\
n_{11} m_{22} - n_{12} m_{21}' = -k. 
\end{split}
\end{equation}

From Step 1, $m_{21}'$ and $m_{22}$ are coprime. Thus, using the Euclidean algorithm, we can construct $n_{11}$ and $n_{12}$ to satisfy Equation \ref{eqn:euclid}. This finishes Step 2.

Thus $M' \in GL_2(\mathbb{Z})$ and $M \equiv M' \mod d$. This completes the proof of the proposition.
\end{proof}

\begin{remark}
\label{rmk:reverse}
As a special case of Lemma \ref{lem:prim_segment}, note that a $d$-primitive segment $E$ can be mapped onto itself in a way that reverses its endpoints.
\end{remark} 
   
Suppose $E$ is a $d$-minimal (not necessarily $d$-primitive) segment. We construct explicitly a canonical form $C(E)$ for $E$. This is a unique representative $d$-minimal segment with the same Weight $\mathbf{W}_d(E)$ as $E$. 

Selct an orientation of $E$ and a residue $i$ so that $W_d(E) = i$ is positive and small as possible. Let $Pr(E)$ denote the unique primitive segment containing $E$.\footnote{We leave routine details of existence and uniqueness to the reader.} Observe that $Pr(E)$ is $k$-primitive for some $k$ dividing $d$. Let $d/k = n$. By the computation in the proof of Lemma 3.15 from \cite{turnerwu}, we observe that 

\begin{equation*}
\begin{split}
W_{k}(Pr(E)) n \equiv i \mod kn \Rightarrow
W_{k}(Pr(E))n = i + skn.
\end{split} 
\end{equation*}

for some $s \in \mathbb{Z}$. Thus $n|i$, and we have $W_k(Pr(E)) = i/n + sk \Leftrightarrow W_k(Pr(E)) \equiv i/n \mod k$. 

Hence, using Lemma \ref{lem:prim_segment}, select a $G$-map $g_E$ mapping $Pr(E)$ to the $k$-primitive segment with endpoints $p_E = (0, i/nk)$ and $q_E = (1/k, i/nk)$. Moreover, select $g_E$ in such a way that $g_E(E)$ is as close to the $y$-axis as possible (it may be necessary to apply Remark \ref{rmk:reverse} to do so). Then the ``canonical form'' $C(E)$ of the edge of Weight $\mathbf{W}_d(E)$ is defined to be $C(E):= g_E(E)$. 

\begin{proposition}
\label{prop:canon_form}
Let $E$ and $F$ be $d$-minimal edges. Then $\mathbf{W}_d(E) = \mathbf{W}_d(F)$ if and only if their canonical forms agree ie $C(E) = C(F)$. 
\end{proposition}

\begin{proof}
This fact is essentially contained in the preceding discussion. The careful reader will note that $C(E)$ is the representative of the $G$-equivalence class $\mathbf{W}_d(E)$ that lies in the first quadrant and is the least distance away from the origin. 
\end{proof}

We summarize the steps of constructing $C(E)$ with an algorithm.

\begin{algorithm}[Detecting when $\mathbf{W}_d(E) = \mathbf{W}_d(F)$]
\label{alg:edge_detect}
\normalfont
Let $E$ be a $d$-minimal segments. The following algorithm constructs the canonical form $C(E)$. 

\begin{enumerate}
\item Select an orientation of $E$ and a positive number $i$ as small as possible such that $W_d(E) = i$. 

\item Construct the unique primitive edge $Pr(E)$ containing $E$. This can be done by methodically evaluating the weight of all $\ell$-minimal segments containing $E$ for $\ell$ dividing $d$. 

\item Suppose $Pr(E)$ is $k$-minimal. Let $n = d/k$. Then using the procedure in the proof of Lemma \ref{lem:prim_segment}, construct a $G$-map $g_E'$ sending $Pr(E)$ to the $k$-primitive minimal segment $F$ from oriented endpoints $p_E = (0, i/nk)$ to $q_E = (1/k, i/nk)$.

\item Let $g_E''$ be the map sending $Pr(E)$ to $F$ with the opposite orientation. The map $g_E''$ can be constructed from $g_E$ using Remark \ref{rmk:reverse}. 

\item Define $C(E)$ to be either $g_E'(E)$ or $g_E''(E)$, whichever is closest to the $y$-axis. 

\end{enumerate}

If $F$ is another $d$-minimal segment, then repeat the above process to construct $C(F)$. Then $C(E) = C(F)$ if and only if $\mathbf{W}_d(E) = \mathbf{W}_d(F)$.

\end{algorithm}

\begin{remark}
\label{rmk:edge_map}
\normalfont
Note that the above procedure also provides maps from $E$ to $F$ for free in the case when $\mathbf{W}_d(E) = \mathbf{W}_d(F)$. Hence, given denominator $d$ polygons $P$ and $Q$ with $d'$-minimal triangulations $\mathcal{T}_1$ and $\mathcal{T}_2$, respectively,  satisfying $\mathbf{W}_d(P) = \mathbf{W}_d(Q)$, we can methodically construct a piecewise $G$-bijection\footnote{Existence of such a map is a consequence of Lemma \ref{lem:Weight_eqn}.} from $\mathcal{T}_1^{-}$ to $\mathcal{T}_2^{-}$.   
\end{remark}

\end{subsection}

\begin{subsection}{Detecting Vertex Compatibility and Mapping}
\label{subsec:vert_compat}

Detecting vertex compatibility of $P$ and $Q$ requires checking that $\mathrm{ehr}_P = \mathrm{ehr}_Q$. Algorithms for doing this were developed by Barvinok \cite{barvinok} have been implemented in the software \texttt{LattE} \cite{LattE}.

Thus, our primary concern is, assuming $\mathrm{ehr}_P = \mathrm{ehr}_Q$, constructing a mapping of the vertices of the triangulation $\mathcal{T}_1$ of $P$ to the vertices of the triangulation $\mathcal{T}_2$ of $Q$. Here is a potential algorithm for doing so.

\begin{algorithm}[Mapping vertices]
\label{alg:vertex_map}
\normalfont
Suppose $P$ and $Q$ are Ehrhart equivalent denominator $d$ polygons with $d'$-minimal triangulations $\mathcal{T}_1$ and $\mathcal{T}_2$, respectively. The following procedure will construct a piecewise $G$-bijection between the vertices of $\mathcal{T}_1$ and $\mathcal{T}_2$. Recall that such a map exists by Lemma \ref{lem:points2}. 

For $n \geq 1$ dividing $d'$:

\begin{enumerate}
\item List the set $S_P$ of $n$-primitive points in $P$ and the set $S_Q$ of $n$-primitive points in $Q$. By Lemma \ref{lem:points1}, each point of $S_P$ is $G$-equivalent to each point of $S_Q$.  

\item  For each $p \in S_P$ and $q \in S_Q$, use the construction from the proof of Lemma \ref{lem:points1} to translate $p$ and $q$ to points $p'$ and $q'$, respectively, visible from the origin.  

\item One can use the Euclidean algorithm to map any two integral points visible from the origin via $GL_2(\mathbb{Z})$ (we leave details to the reader).\footnote{Hint: Try to map the point $(1,0)$ to a general integer point.} Thus one can construct a $G$-map between $p'$ and $q'$ with this procedure.  
\end{enumerate}
\end{algorithm}

\end{subsection}

Finally, using our main theorem, Theorem \ref{thm:dehn}; along with a combination of Algorithms \ref{alg:facet_map}, \ref{alg:edge_detect}, and \ref{alg:vertex_map}; Remark \ref{rmk:edge_map}; and Observation \ref{obs:facet_map} we can detect and/or construct rational equidecomposability relations between denominator $d$ polygons $P$ and $Q$. This provides our response to Question \ref{que:alg} posed by Haase--McAllister in \cite{haase}. 
\end{section}

\begin{section}{Further Questions}
\label{sec:questions}

We briefly summarize some questions and directions for further inquiry, restricting to the case of polygons in accordance with the style of this paper.  

\begin{enumerate}

\item Are any of the criteria in Theorem \ref{thm:dehn} redundant? Moreover, when checking facet compatibility between denominator $d$ polygons $P$ and $Q$, can the requirement of having to check $\mathcal{D}_{d'}(P) = \mathcal{D}_{d'}(Q)$ for all $d'$ divisible by $d$ be reduced to a finite check? We conjecture the answer is ``yes,'' and that it suffices to only check $\mathcal{D}_{d}(P)$ and $\mathcal{D}_{d}(Q)$ to determine facet compatibility. This would provide a more satisfying answer to Question \ref{que:dehn} posed by Haase--McAllister \cite{haase}. 

\smallskip

\item Is it possible, using the algorithmic framework of Section \ref{sec:algorithm}, to develop computationally \emph{efficient} procedures for detecting/constructing equidecomposability relations? Even detecting facet map equivalence is an interesting problem (see Section \ref{subsec:facet_compat}). This would provide a more complete answer to Question \ref{que:alg} orginally posed by Haase--McAllister \cite{haase}. 

\end{enumerate}

\end{section}

\begin{section}{Acknowledgments}
This research was conducted during Summer@ICERM 2014 at Brown University and was generously supported by a grant from the NSF. First and foremost, the authors express their deepest gratitude to Sinai Robins for introducing us to the problem, meeting with us throughout the summer to discuss the material in great detail, suggesting various useful approaches, answering our many questions, and critically evaluating our findings. In the same vein, we warmly thank our research group's TAs from the summer, Tarik Aougab and Sanya Pushkar, for numerous beneficial conversations, suggestions, and verifying our proofs. In addition, we thank the other TA's, Quang Nhat Le and Emmanuel Tsukerman, for helpful discussions and ideas. We also thank Tyrrell McAllister for visiting ICERM, providing an inspiring week-long lecture series on Ehrhart theory, and discussing with us our problems and several useful papers in the area. We are grateful to Jim Propp for a very long, productive afternoon spent discussing numerous approaches and potential invariants for discrete equidecomposability. We also thank Hugh Thomas, who is a professor at the second author's home university, for interesting discussions as well as carefully reviewing and commenting on drafts of these results. We extend our gratitude to Michael Mossinghoff and again to Sinai Robins for coordinating the REU. Finally, we thank the ICERM directors, faculty, and staff for providing an unparalleled research atmosphere with a lovely view.
\end{section}

\bibliographystyle{alphaurl}
\bibliography{disc_equi_bib}

\end{document}